\documentclass[11pt]{article}

\usepackage[a4paper]{geometry}
\usepackage[utf8]{inputenc}
\usepackage{amsmath, amsthm, amssymb, amsfonts}
\usepackage{parskip}
\usepackage{mathrsfs}
\usepackage{mathtools}
\usepackage{graphicx}
\usepackage{xtab}
\usepackage{xspace}
\usepackage[shortlabels,inline]{enumitem}
\usepackage{lipsum}
\usepackage{authblk}
\usepackage{tgcursor}

\usepackage{hyperref}
\hypersetup{
    colorlinks,
    citecolor=black,
    filecolor=black,
    linkcolor=black,
    urlcolor=black
}
\usepackage{cleveref}

\newcommand{\bx}{\mathbf{x}}
\newcommand{\bs}{\mathbf{s}}

\newcommand{\bc}{\mathbf{c}}
\newcommand{\boldm}{\mathbf{m}}

\newcommand{\affa}[1]{\widetilde{A}_{#1}}

\DeclareMathOperator{\Res}{Res}

\DeclareMathOperator{\Supp}{Supp}
\DeclareMathOperator{\mult}{mult}

\DeclareMathOperator{\Real}{Re}

\DeclareMathOperator{\height}{ht}
\newcommand{\comment}[1]{}
\newcommand{\sym}[2]{\left(\frac{#1}{#2}\right)}

\usepackage{parskip}
\usepackage{graphicx}
\usepackage[nodayofweek]{datetime}
\usepackage{tikz}
\usepackage{tikz-cd}
\usepackage{setspace}
\usepackage[normalem]{ulem}

\newdateformat{monthyear}{\monthname[\THEMONTH], \THEYEAR}

\newtheoremstyle{named}{}{}{\itshape}{}{\bfseries}{.}{.5em}{\thmnote{#3's }#1}

\theoremstyle{definition}
\newtheorem{definition}{Definition}[section]
\newtheorem{remark}[definition]{Remark}
\newtheorem*{remark*}{Remark}

\theoremstyle{plain}
\newtheorem{theorem}[definition]{Theorem}
\newtheorem*{theorem*}{Theorem}
\newtheorem{corollary}[definition]{Corollary}
\newtheorem{proposition}[definition]{Proposition}
\newtheorem{lemma}[definition]{Lemma}

\newtheorem{hypothesisH}{Hypothesis}

\theoremstyle{remark}

\newtheorem*{example*}{Example}

\def\hi{\mathfrak{h}}
\def\la{\langle}
\def\ra{\rangle}
\def\Zc{\mathcal{Z}}

\def\Pc{\mathcal{P}}
\def\om{\omega}
\def\F{\mathbb{F}}
\def\C{\mathbb{C}}
\def\R{\mathbb{R}}
\def\Z{\mathbb{Z}}
\def\bm{\mathbf{m}}
\def\bt{\mathbf{t}}

\def\bh{\mathbf{h}}
\def\bd{\mathbf{d}}
\def\wbd{\widehat{\mathbf{d}}}
\def\wlam{\widehat{\lam}}
\newcommand{\Ec}{\mathcal E}
\def\lam{\lambda}
\def\a{\alpha}
\def\CG{\mathrm{CG}}
\def\ss{\sigma}
\def\dd{\delta}
\def\wD{\widetilde{D}}
\def\wDD{\widetilde{\Delta}}
\def\wA{\widetilde{A}}

\def\wZ{\widetilde{Z}}
\def\wF{\widetilde{F}}
\def\wZc{\widetilde{\mathcal{Z}}}

\renewcommand{\Re}{\mathrm{Re}}
\def\abs{\mathrm{abs}}
\def\XC{X_\mathbb{C}}
\def\symb#1#2{\left(\frac{#1}{#2}\right)}
\def\be{\begin{equation}}  \def\ee{\end{equation}}

\def\qpoc#1#2{\left(#1; #2 \right)_\infty}
\def\ov#1{\overline{#1}}

\def\geom{\mathrm{geom}}

\def\re{\mathrm{re}}
\def\im{\mathrm{im}}
\def\hgt{\mathrm{ht}}
\def\ve{\varepsilon}
\def\Zm{\mathsf{Z}}

\begin{document}
\title{A decomposition of Weyl group multiple Dirichlet series\\
for symmetrizable Kac-Moody root systems}
\author{Alexandru A. Popa, Jack Walsh}
\newcommand{\Addresses}{{
\bigskip
\noindent\textsc{Alexandru A. Popa}\\
Institute of Mathematics ``Simion Stoilow'' of the Romanian Academy\\
Calea Grivitei 21, Bucharest, Romania\\
\textit{Email address}: \texttt{aapopa@gmail.com}

\smallskip
\noindent\textsc{Jack Walsh}\\
Institute of Mathematics ``Simion Stoilow'' of the Romanian Academy\\
Calea Grivitei 21, Bucharest, Romania\\
\textit{Email address}: \texttt{jackwalsh23@hotmail.co.uk}
}}

\maketitle

\begin{abstract}
We study twisted Weyl group multiple Dirichlet series attached to symmetrizable
Kac-Moody root systems, using the Chinta-Gunnells method to construct their $p$-parts.
Our main result is a decomposition theorem for functions invariant under the twisted
Chinta-Gunnells action: under natural analytic hypotheses, such a function has a unique
expansion in terms of shifted Chinta-Gunnells averages, indexed by an explicit set
of dominant weights determined by the twisting parameter.
In particular, we show that this decomposition holds for twisted multiple Dirichlet series over rational
function fields. For finite root systems, these results were proved by Friedlander~\cite{Friedlander2018TwistedWG}.
We also show that the relevant Chinta-Gunnells averages,
modulo an infinite product over the imaginary roots,  admit
analytic continuation to the interior of the complexified Tits cone.

In the quadratic affine $\widetilde{A}_1$ case, we prove extra functional equations,
not arising from the Weyl group, for the untwisted average and
for averages twisted by fundamental weights. As a consequence,
we obtain an explicit formula for a normalized multiple Dirichlet series
with coprime, square-free twisting parameters,
and show that it also satisfies an extra functional equation.

\end{abstract}

\section{Introduction}

A Weyl group multiple Dirichlet series (WMDS) associated with a finite
root system $\Phi$ of rank $r$ is a Dirichlet series in $r$ complex variables over a global field, satisfying a group of functional equations isomorphic to the Weyl group of $\Phi$.
Such series were first considered by Brubaker, Bump, Chinta, Friedberg, and Hoffstein in a series of papers \cite{wmds1, wmds2, wmds3, wmds4}, in order to unify many objects that had been studied previously on a case-by-case basis. According to the \emph{Eisenstein conjecture}, proposed in~\cite{wmds1}, they naturally appear as Whittaker coefficients of an Eisenstein series on a global metaplectic group with root system
$\Phi$. For  finite root systems, the Eisenstein conjecture was recently proved
by Chen~\cite{Chen2024MetaplecticMDS}, after being proved in type $A$ by  Brubaker, Bump, and Friedberg \cite{BrubakerBumpFriedberg2011Crystal}, and in type $B$ by Friedberg and Zhang \cite{FriedbergZhang2013OddOrthogonal}.

Historically, one of the main reasons for introducing such series was to extract arithmetic information about  families of automorphic $L$-functions, such as nonvanishing of central values or the asymptotics of moments of quadratic twists of automorphic $L$-functions~\cite{GoldfeldHoffstein1985Eisenstein, GoldfeldHoffstein1990Nonvanishing}.
A central example is provided by the WMDS associated with moments of
the family of quadratic Dirichlet $L$-functions. In this setting,  the analytic properties of the series--its meromorphic continuation and the residues at its poles--govern the corresponding moment asymptotics. For a survey of what is known about moments of quadratic Dirichlet $L$-functions, we refer to the introduction
of~\cite{diaconu2023quadraticweylgroupmultiple}. The WMDS associated with the $r$-th moment in this family is
attached to a root system $\Phi_r$ of rank $r+1$, whose Dynkin diagram is star-shaped with a central node connected to $r$ exterior nodes. This root system is finite dimensional for $r\le 3$,
affine for $r=4$, hyperbolic Kac-Moody if $r=5$, and indefinite Kac-Moody for $r\ge 6$.

In this paper we are concerned with twisted WMDS for symmetrizable Kac-Moody root systems. We now describe their construction, specializing to the function field setting considered in this paper.
 Fix an integer $n\ge 2$ and a rational function field $\mathbb{F}_q(t)$ with $q \equiv 1 \mod{2n}$. Let $\Phi$ be a root system of rank $r$ and fix a twisting parameter $\bc = (c_1, \dots, c_r)$ with monic polynomials $c_i\in\mathbb{F}_q[t]$. Then the WMDS attached to $\Phi$ is given by
\begin{align*}
    \Zc(\bs;\bc) = \sum_{\boldm}\frac{H(\boldm;\bc)}{|m_1|^{s_1}\cdots |m_r|^{s_r}}
\end{align*}
where $\bs = (s_1,\dots, s_r) \in \mathbb{C}^r$ with sufficiently large real parts,
$|m|=q^{\deg m}$, and the sum is over all tuples of monic polynomials $\boldm = (m_1, \dots, m_r) \in (\mathbb{F}_q[t])^r$. The coefficients $H$ satisfy a twisted multiplicativity property,
involving the $n$-th order residue symbols and depending on the Dynkin diagram of $\Phi$
\cite{Chinta2006OnTP,Chinta2008ConstructingWG},
so that the series is determined uniquely by specifying its $\pi$-parts. These are the generating series
\begin{align*}
    \sum_{(n_1,\dots, n_r) \in \mathbb{Z}_{\ge 0}^r}H(\pi^{n_1}, \dots , \pi^{n_r};\pi^{v_\pi(c_1)},\dots,\pi^{v_\pi(c_r)})|\pi|^{-n_1s_1}\cdots|\pi|^{-n_rs_r}
\end{align*}
for a monic irreducible polynomial $\pi \in \mathbb{F}_q[t]$, with $v_\pi(c)$ the exponent of $\pi$ in the factorization of $c$.
While there are several equivalent ways to represent these $\pi$-parts, we focus on the Chinta-Gunnells ``averaging method'' \cite{Chinta2006OnTP, Chinta2008ConstructingWG, Chinta2006WeylGM}. That is, Chinta and Gunnells defined
an action of the Weyl group of $\Phi$ on rational functions, depending on the twisting parameters
$(v_\pi(c_1), \dots, v_\pi(c_r) )$, and defined the $\pi$-part in
terms of the sum of Weyl group elements acting on the function $1$ (see Definition~\ref{def:cwaverage}).
By construction, the $\pi$-part is invariant under this action, and it turns out that this invariance is inherited by the  WMDS
$\Zc$ after a change of variables, with respect to the twisting parameter $(\deg c_1, \dots, \deg c_r)$. For finite $\Phi$,
this was established by Chinta and Gunnells
over number fields \cite{Chinta2008ConstructingWG},
and by Friedlander over function fields \cite{Friedlander2018TwistedWG}.

The study of WMDS for symmetrizable Kac-Moody root systems was initiated by Lee and Zhang \cite{Lee2012WeylGM},
using the Chinta-Gunnells method to construct the $\pi$-parts.
Patnaik and Puskas constructed metaplectic covers of
Kac-Moody groups over nonarchimedean local fields and proved a
Casselman-Shalika formula for their Whittaker functions in terms of Chinta-Gunnells averages
\cite{Patnaik2017MetaplecticCO}. This provides important local evidence for the Eisenstein conjecture in this
setting: after the correct normalization, one expects the global WMDS to occur as Whittaker coefficients of metaplectic Kac-Moody Eisenstein series.
For finite $\Phi$, the Casselman-Shalika formula was proved by McNamara~\cite{McNamara2016CasselmanShalika}, and it was a key ingredient in the proof of the Eisenstein conjecture in~\cite{Chen2024MetaplecticMDS}.

In this paper we are concerned with the following question:
to what extent is a function--with $\Zc(\bs;\bc)$ as a motivating example--determined
by being invariant under the Chinta-Gunnells action? To state what is known, we need to
introduce some notation, with more details given in Section~\ref{sec:prelim}.
 Let $\alpha_i$ denote the simple roots in $\Phi$, and $\alpha_i^\vee$ the simple coroots.
 We denote the sets of real and imaginary roots by $\Phi_\re$ and $\Phi_\im$, respectively,
and the multiplicity of a root $\a$ by $\mult(\a)$.
It is convenient to encode a twisting parameter $(l_1, \ldots, l_r)\in\mathbb{Z}_{\ge 0}^r$ by a dominant weight $\omega\in P^+$,
such that $\la \omega, \alpha_i^\vee \ra=l_i$, with $\la\cdot,\cdot\ra$ a Weyl group invariant pairing.
Let $Z_\omega(\bx;q)$ be the corresponding Chinta-Gunnells
average, where $\bx=(x_1,\ldots, x_r)$ is  a tuple of variables.  By construction, $Z_\omega$
is invariant under the Chinta-Gunnells
action $|_\omega^\CG \sigma$ for all $\sigma$ in the Weyl group $W$.
Let
\[
 D_\re(\bx;q) = \prod_{\alpha \in \Phi_{\re}^+}(1-q\bx^{m_\alpha\alpha}),\quad
 D_\im(\bx;q) = \prod_{\alpha \in \Phi_{\im}^+}(1-q\bx^{m_\alpha\alpha})^{\mult(\a)}
\]
and $D=D_\im D_\re$,
where $m_\alpha\mid n$ are integers used to define the action, and for
$\lambda=\sum k_i\alpha_i$ in the root lattice $Q$ we set $\bx^\lambda=\prod x_i^{k_i}$.
Note that $\mult(\a)=1$ for $\a\in\Phi_\re$.

When $\Phi$ is finite and $n=2$, it was shown by Chinta, Friedberg, and Gunnells
that any rational function $Z(\bx)$ invariant under the action above and having polynomial numerator
$N=D_\re Z=\sum_{\lambda\in Q^+} b_\lambda \bx^\lambda$
is determined uniquely by the coefficients $b_{\omega-\xi}$ for $\xi$ in the set
\begin{equation}\label{eq:Pi_omega}
 \Pi_\omega=\{\xi\in P^+\mid \xi\le \omega  \}
\end{equation}
of dominant weights in the representation with highest weight $\omega$ \cite{Chinta2006OnTP}. Later
Friedlander generalized this result to arbitrary $n$, and further proved that the space of such functions
has dimension precisely the cardinality of $\Pi_\omega$, with $\bx^{\omega-\xi} Z_\xi$
for $\xi\in \Pi_\omega$ forming a basis \cite{friedlanderpparts}. The proof relies on
encoding the invariance under simple reflections as a set of relations satisfied by the coefficients $b_\lambda$ of $N$, and showing that the relations have a recursive structure depending on the Weyl group orbits centered at $\omega+\rho$ with $\rho$ the sum of the fundamental weights.
For a concise overview of this argument in the quadratic case see \cite[Sec. 3]{DIACONU2025110359}.

Our main result extends this finite-type picture to the case of symmetrizable Kac-Moody root systems.
Let $Z(\bx)$ be invariant under the
Chinta-Gunnells action $|_\omega^{\mathrm{\CG}}$, and assume that $D_\re Z$ is
holomorphic on a certain $W$-invariant subset of the complexified Tits cone after the change of variables $x_i=q^{-s_i}$
--see Hypothesis~\ref{hyp:analyticassumptions} for the details.
 We assume that $Z$ has a Taylor expansion around the origin,
$$
Z(\bx)=\sum_{\lambda\in Q^+}a_\lambda \bx^\lambda,
$$
and we show that the coefficients $a_\lambda$ satisfy relations  given explicitly in Proposition~\ref{prop:relations},
expressing the invariance of $Z$ under simple reflections.
From these relations, an immediate recursive argument in terms of the height of
$\lambda$ shows that all the coefficients are determined by the coefficients $a_{\omega-\xi}$
with $\xi$ in the set $\Pi_\omega$ defined in~\eqref{eq:Pi_omega} (Corollary~\ref{cor:coeffsdefine}).
We note that working with the coefficients of the Taylor expansion
of $Z$, rather than those of its ``numerator'' $D_\re Z$,
circumvents the study of the recursive structure of Weyl group orbits obtained for finite $\Phi$
in \cite{Chinta2006OnTP, friedlanderpparts}; instead, our recursion uses only the height
of $\lam\in Q$.

Moreover, we show in Theorem~\ref{the:basis} that the functions   $\bx^{\omega-\xi} Z_\xi$
for  $\xi\in \Pi_\omega$ play the role of a basis for the space of solutions: there is a unique decomposition
\begin{equation}\label{eq:decomp_intro}
Z(\bx)=\sum_{\xi \in \Pi_\omega} c_{\omega-\xi} \bx^{\omega-\xi} Z_\xi(\bx;q)
\end{equation}
that converges normally in a polydisk,
where $c_\lambda$ are the coefficients of the Taylor expansion of $\Delta Z$. Here
$\Delta=\Delta_\re\Delta_\im$ is the infinite product defined as $D$
with each factor $1-q \bx^{m_\a \a}$ replaced by $1-\bx^{m_\a \a}$.
We also show that the Chinta-Gunnells averages appearing in this decomposition
extend to the same domain as the Weyl-Kac character, modulo the infinite product $\Delta_\im$:
the function $D_\re\Delta_\im Z_\xi$ admits analytic continuation
to the interior of the complexified Tits cone (Proposition~\ref{prop:convergence}).
For the untwisted quadratic average, this was previously known in the simply-laced affine case~\cite[Proposition~3.1]{diaconu2023quadraticweylgroupmultiple}.

At the global level, we show in Theorem~\ref{thm:fullglobalconvergence} that the WMDS $\Zc(\bs;\bc)$ extends
meromorphically to an explicit Weyl group invariant subset of the Tits cone, and it is invariant under the Chinta-Gunnells action.
This gives an analogue of the number field result in~\cite{Lee2012WeylGM}; we give a more direct proof
relying on the fact that the Kubota series over function fields is a rational function.
Corollary~\ref{thm:globalbasis} then gives the decomposition~\eqref{eq:decomp_intro} for the WMDS $\Zc(\bs;\bc)$.
The ultimate goal would be to show that the WMDS itself has meromorphic continuation to the interior
of the complexified Tits cone, but that problem seems extremely difficult.
For the star-shaped root systems $\Phi_r$ and $n=2$, this was conjectured by Diaconu and Twiss \cite{DiaconuTwiss2023Secondary}.
It was shown there that the conjectured analytic continuation of the untwisted quadratic WMDS
and of certain twists related to the ones considered here would lead to asymptotic formulas
for the $r$-th moment of quadratic Dirichlet $L$-functions, with infinitely many secondary terms.
In the case of $\Phi_4$, the analytic continuation of the \emph{untwisted} quadratic WMDS
over function fields was proved in~\cite{diaconu2023quadraticweylgroupmultiple}, leading to an asymptotic formula with infinitely many secondary terms for a \emph{weighted} fourth moment.
We hope the results of this paper will shed some light on the more difficult problem of establishing
analytic continuation for the \emph{twisted} quadratic WMDS.

The decomposition of $\Zc(\bs;\bc)$ into Chinta-Gunnells
averages therefore suggests the following problem,
whose resolution would have implications for its meromorphic continuation.

\noindent\textbf{Problem.} Determine explicit formulas, or  effective
growth estimates, for the coefficients appearing in the expansion \eqref{eq:decomp_intro}
for  $\Zc(\bs;\bc)$, under the change of variables $x_i=q^{-s_i}$.

To shed some light on this problem, we remark that
one complicating feature in the infinite-dimensional case is the failure of the local-to-global principle
in the untwisted situation, namely for $\Zc(\bs):=\Zc(\bs;(1,\ldots,1) )$. For finite $\Phi$,  the above decomposition shows that $\Zc(\bs)$ equals the untwisted Chinta-Gunnells average $Z_0(\bx;q)$, after the change
of variables $x_i=q^{-s_i}$, and with compatible definitions of the other parameters used to define the action.
This is known as the local-to-global principle\footnote{Our normalization differs from the standard one,
so the local-to-global principle is stated differently from \cite{Chinta2008ConstructingWG, friedlanderpparts, Lee2012WeylGM}.}.

For general $\Phi$, the situation is more complicated due to the existence of imaginary roots.
We show in Corollary~\ref{cor:coeffsdefine} that a function $Z$ satisfying the above analytic hypotheses and
invariant under the untwisted action
is completely determined by its Taylor coefficients $a_\lambda$ for $\lambda \in Q^+$
such that
\begin{equation}\label{eq:lambda_domain}
 \la \lambda, \alpha_i^\vee \ra\le 0   \qquad  (i=1,\dots, r).
\end{equation}
The set of such nonzero $\lambda$ with connected support in the Dynkin diagram of $\Phi$
is precisely the set of representatives for the orbits of imaginary roots $\Phi^+_{\text{im}}$
under the Weyl group described by Kac in \cite[\S 5.3]{Kac_1990}. Since the precise
determination of imaginary roots is a difficult problem beyond the affine case, making progress on the above
problem will require new ideas.

For irreducible affine $\Phi$, the nonzero $\lambda\in Q^+$ satisfying~\eqref{eq:lambda_domain}
are precisely the multiples of the minimal positive imaginary root,
and we recover the result proved for $\wD_4$ and $n=2$ in~\cite{BucurDiaconu2010Moments},
in connection with the fourth moment of quadratic Dirichlet $L$-functions.
More precisely, the decomposition~\eqref{eq:decomp_intro} implies that the
untwisted WMDS $\Zc(\bs)$ satisfies
\[\Zc(\bs)= C(\bx^\delta) Z_0(\bx;q), \qquad \text{for } x_i=q^{-s_i},
\]
where $C(x)$ is a power series in one variable and $\delta$ is the minimal positive imaginary root.
For $\wD_4$, the problem above was solved in \cite{diaconu2023quadraticweylgroupmultiple}. More precisely,
if one normalizes $Z_0$ by
\[\widetilde{Z}_0(\bx;q) = \frac{Z_0(\bx;q)}{(q \bx^\delta; \bx^{2\delta})^2_\infty}
\]
with $(a;q)_\infty:=\prod_{n\ge 0}(1-a q^n)$ the infinite Pochhammer symbol,
and defines $\widetilde{\mathcal{Z}}(\bs)$ as the MDS with $\pi$-part $\widetilde{Z}_0(|\pi|^{-\bs};|\pi|^{-1})$,
then the local-to-global principle is recovered:
\[
\widetilde{\mathcal{Z}}(\bs)= \widetilde{Z}_0(q^{-\bs};q),
\qquad q^{-\bs}=\big(q^{-s_1},\ldots, q^{-s_r}\big).
\]
The key ingredient used in the proof is an unexpected symmetry of the
untwisted average $Z_0$,  under the transformation $\sqrt q\mapsto \sqrt q \bx^\delta$.
The existence of this extra functional equation--and therefore of the local-to-global identity after a suitable normalization--was proved in an ad hoc way for $\wD_4$, but it is expected to hold for all affine root systems \cite{diaconu2023quadraticweylgroupmultiple}.

For the twisted averages, little is known about the existence of
such an extra functional equation for affine root systems. In the last section, we discuss this question
in the case $n=2$ for $\wA_1$. We prove the existence of an extra functional equation
both for the untwisted average $Z_0$ (Subsection~\ref{subsec:untwistedaffa1computation}) and for the twisted $Z_{\omega_i}$, where $\omega_i$ is a fundamental weight (Subsection~\ref{subsec:twistedaffA1}). As a consequence, we prove explicit formulas
for the Chinta-Gunnells averages $Z_0$ and $Z_{\omega_i}$, as well as for a normalized MDS
$\wZc(\bs; c_0,c_1)$ constructed in Subsection~\ref{sec:a1globalcomputations}
with square-free, coprime twisting polynomials $c_0, c_1$--see Theorems~\ref{the:explicituntwistedaverage},
\ref{the:explicittwistedaverage} and \ref{thm:explicitaffa1mds}. The formula for $Z_0$ can be seen as a $u$-deformation of the classical Jacobi triple product formula:
\[
\sum_{n\in\mathbb{Z}} u^nx_0^{n}\bx^{\frac{n(n-1)}{2}\delta}\frac{(u^{-1}x_0;\bx^{\delta})_n}{(ux_0;\bx^{\delta})_n} = \frac{(u^2\bx^{\delta};\bx^{2\delta})_\infty(x_0^2;\bx^{2\delta})_\infty(x_1^2;\bx^{2\delta})_\infty(\bx^{2\delta};\bx^{2\delta})_\infty}{(ux_0;\bx^{\delta})_\infty(ux_1;\bx^{\delta})_\infty},
\]
where $\delta=\alpha_0+\alpha_1$ and $u^2= q$; for $u=-1$ it becomes the Jacobi triple product identity. The left-hand side can
be written as the basic hypergeometric function  ${}_1\psi_2\!\left(\begin{matrix}a\\ b,0\end{matrix};p,-b\right)$, for $p=\bx^\dd$, $a=u^{-1} x_0$, $b=ux_0$, and the identity above can also be proved using standard hypergeometric function transformations.

Interestingly, the MDS $\wZc(\bs; c_0,c_1)$ also satisfies an extra functional equation, although
its coefficients may not be polynomials in $\sqrt q$. Instead, its coefficients can be expressed
in terms of $q$-Weil numbers coming from the coefficients of the Dirichlet $L$-functions $L(s,\chi_{c_i})$,
and the extra functional equation involves the transformation $\alpha\mapsto \alpha q^{- m \delta (\bs)}$,
where $\alpha$ is a $q$-Weil number of weight $m$--see Theorem~\ref{thm:tauactiononmds}.

The case of $\wA_1$ is quite degenerate,
and we certainly do not expect such explicit formulas to hold for other affine root systems. However,
this case suggests that the global MDS may satisfy extra functional equations in the twisted case as
well, which would be an essential ingredient in proving a local-to-global principle in the twisted affine setting.

Finally, we mention a different construction of a canonical untwisted MDS over the rational
function field for arbitrary $\Phi$, using a set of axioms that includes the local-to-global principle.
The axiomatic approach was initiated
by Diaconu and Pa\c sol for star-shaped root systems in \cite{Diaconu2018ModuliOH},
further developed by
Whitehead for root systems of type $\affa{r}$~\cite{Whitehead2014MultipleDS}, and later generalized by
Sawin~\cite{Sawin2024GeneralMD} to a setting that includes the WMDS associated with symmetrizable Kac-Moody root systems.
The axiomatic approach departs from the traditional one, in that the WMDS is not assumed to satisfy a group
of functional equations. However,
Sawin and Whitehead showed that the resulting WMDS satisfies functional equations that are equivalent to invariance under the untwisted Chinta-Gunnells action in the symmetrizable Kac-Moody case \cite{sawinwhitehead}.
Therefore, the decomposition~\eqref{eq:decomp_intro}  makes sense for the axiomatic MDS,
and it is an interesting question to investigate the coefficients appearing in this decomposition.
In the $\wD_4$ case it was shown in~\cite{diaconu2023quadraticweylgroupmultiple} that the WMDS constructed with normalized
$\pi$-part $\widetilde{Z}_0$ above agrees with the axiomatic MDS.
For star-shaped root systems $\Phi_r$ with $r>4$, relating the axiomatic MDS to the one
constructed from $\pi$-parts is an interesting open problem, in view of its applications to the
asymptotics of $r$-th moments of Dirichlet $L$-functions.

 \textbf{Acknowledgments.} The first author is grateful to Adrian Diaconu, Vicen\c tiu Pa\c sol, and
 Bogdan Ion for many helpful conversations.
 The authors were supported by the project CF159/31.07.2023, ``Group schemes, root systems, and related representations'', funded by the European Union-NextGenerationEU through the PNRR call no. PNRR-III-C9-2023-I8.

\section{Preliminaries}\label{sec:prelim}
Let $A$ be a symmetrizable generalized Cartan matrix of size $r$, with realization $(\mathfrak{h},\Pi,\Pi^\vee)$, root system $\Phi$, and Weyl group $W$; we refer to \cite{Kac_1990} for
general background on Kac-Moody algebras.
We set $\Pi = \{\alpha_1,\dots, \alpha_r\}\subset \hi^*$ and $\Pi^\vee = \{h_1, \ldots, h_r\}\subset \hi$  to be the sets of simple roots and simple coroots, respectively, so that $A=(\a_j(h_i))_{1\le i,j\le r}$.
The Weyl group is generated by simple reflections $\ss_i$, acting on $\hi^*$ and by duality on $\hi$.

We fix a decomposition $A=DB$, with a diagonal matrix $D$ and a symmetric matrix $B$. We normalize
the matrix $B=(b_{ij})$ so that $b_{ij}\in \Z$ and $b_{ii}\in 2\Z$.
We fix a symmetric, non-degenerate, Weyl group invariant bilinear form
$\la \cdot, \cdot \ra$ on $\hi^*$,
such that $ \la\a_i,\a_j \ra=b_{ij}$ \cite[Ch. 2]{Kac_1990}. Let $q(\a)=\tfrac12 \la\a,\a\ra$ be the associated
Weyl invariant quadratic form, which takes integral values with our
normalization.\footnote{Our matrix $B$, and therefore the associated bilinear form, are twice the ones in~\cite{Lee2012WeylGM}.} Using the bilinear form, we associate to the simple coroots $h_i$
the vector $\a_i^\vee=\a_i/q(\a_i)\in \hi^*$, such that for $\lam\in \hi^*$ we have
\[
\lam(h_i)=\la\lam, \a_i^\vee \ra.
\]
More generally, for any $\a\in\Phi_{\re}$ we define $\a^\vee=\a/q(\a)$.

Let $Q=\bigoplus_{i=1}^r\Z\alpha_i$ be the root lattice, let
$Q^+=\sum_{i=1}^r\Z_{\geq0}\alpha_i$, and for $\lam,\mu\in\hi^*$ write
$\lambda\leq\mu$ if $\mu-\lambda\in Q^+$. For
$\lambda=\sum_i k_i\alpha_i\in Q$, put
$\height(\lambda)=\sum_i k_i$.
Let
$$P=\{\lambda\in\hi^* \mid \la\lambda,\alpha_i^\vee \ra\in \mathbb{Z},\ i=1,\ldots, r \}$$
be the weight lattice, and
denote by $P^+$ and $P^{++}$, respectively the subsets of dominant and strongly dominant weights,
defined by inequalities $\la\lambda,\alpha_i^\vee \ra\in \mathbb{Z}_{\ge 0}$ and
$\la\lambda,\alpha_i^\vee \ra\in \mathbb{Z}_{> 0}$ respectively.

A root $\alpha \in \Phi$ is called real if $\alpha = \sigma \alpha_i$ for some $\sigma \in W$,  and it is called imaginary if it is not real. The sets of real and imaginary roots are denoted by $\Phi_{\re}$ and $\Phi_{\text{im}}$, respectively. Let $\mult(\alpha)$ denote the multiplicity of $\alpha$, i.e. the dimension of the corresponding root space.

For $\sigma\in W$, we let $\Phi(\sigma):=\{\lambda \in \Phi^+\mid \sigma(\lambda) \in \Phi^-\}$. This set consists
of real roots, and we denote its cardinality by $\ell(\sigma)$; this is the same as the length of
$\sigma$ in terms of the simple reflections generating $W$.

Fix an integer $n\ge 2$. To define the Chinta-Gunnells action,
for each $\alpha \in \Phi$ we define an integer $m_\a\mid n$ by
\begin{align*}
m_\alpha =
\begin{cases}
\frac{n}{\gcd(n,q(\alpha))} &\text{if }\alpha \in \Phi_{\re},\\
n &\text{if } \alpha \in \Phi_{\text{im}}.
\end{cases}
\end{align*}
We write $m_i$ for $m_{\alpha_i}$.

We fix a real number $q>1$. Set  $g_0 = -1$, and let $g_1, \dots, g_{n-1}\in \C$.
For arbitrary $t\in \Z$, we set $g_t=g_{r_n(t)}$ where
$0 \leq r_{k}(t)\leq k-1$ is the remainder of $t$ modulo $k$.
We assume that the $g_t$ satisfy the relations
\begin{align*}
g_tg_{-t} = q \quad \text{ if } t\not\equiv 0 \mod n.
\end{align*}

With respect to the basis $\Pi$ of simple roots we write $x_i = \bx^{\alpha_i}$ and regard $\bx = (x_1, \dots, x_r)$ as a multivariable. For $\lambda=\sum_i k_i \alpha_i\in Q$, we denote by
$\bx^{\lambda}$ the monomial $\prod_i x_i^{k_i}$.
There is an action of the Weyl group $W$ on monomials given by $\sigma \bx^{\lambda} = \bx^{\sigma^{-1}\lambda}$.

We fix twisting parameters $l_1,l_2, \dots, l_r \in \mathbb{Z}_{\geq 0}$.
It is convenient to encode the twisting parameters by fixing  a
dominant weight $\omega\in P^+$  such that
\begin{equation}\label{eq:twist_parameter}
 \la\omega,\alpha_i^\vee \ra=l_i, \quad  (i=1,\ldots, r).
\end{equation}
Note that if $A$ is nonsingular, we have a unique choice $\omega = \sum_{i=1}^r l_i\omega_i$,
with $\omega_i\in P^+$ the fundamental weights that satisfy $\la \om_i, \a_j^\vee\ra =\dd_{ij}$.
If $A$ is singular, the fundamental weights are not uniquely determined by these relations
and we fix a choice, e.g. as in~\cite[Sec. 1]{Lee2012WeylGM}.
\begin{definition}
The \textit{twisted Chinta-Gunnells action} with twisting parameter $\omega$ as in~\eqref{eq:twist_parameter}
is defined on monomials as
\begin{align*}\bx^\lambda \vert^{\CG}_\omega\sigma_i = \frac{\bx^{\sigma_i\lambda + (l_i+1)\alpha_i}}{1-qx_i^{m_i}}\left[x_i^{-r_{m_i}(\delta_i(\lambda))}(1-q)-(x_i^{-m_i}-1)g_{q(\alpha_i)\delta_i(\lambda)} \right],
\end{align*}
where $\delta_{\omega,i}(\lambda)=\delta_i(\lambda):= l_i+1-\langle \lambda, \alpha_i^{\vee} \rangle$.

We also have the related action $\bx^\lambda\vert_\omega\sigma_i = -x_i^{m_i}\cdot\bx^\lambda\vert_\omega^{\CG}\sigma_i$.
\end{definition}

It was shown by Chinta and Gunnells \cite{Chinta2006WeylGM, Chinta2008ConstructingWG} that these actions extend to right actions of
$W$ on the space of rational functions in $\bx$. For infinite root systems,
Lee and Zhang \cite{Lee2012WeylGM} defined these actions termwise on a space of formal distributions $\sum_{\lambda\in Q} a_\lambda \bx^\lambda$
with a growth condition on the coefficients $a_\lambda$,
which ensures that this space is preserved by the action. Here we take the analytic viewpoint,
by showing that the Chinta-Gunnells averages and their WMDS counterparts extend meromorphically
to certain $W$-invariant subsets of the Tits cone after the change of variables $x_i=q^{-s_i}$.
It is therefore convenient to give the following equivalent definition of the Chinta-Gunnells action on more general functions, mirroring the one given in the quadratic case in~\cite{Chinta2006OnTP}.

\begin{definition}\label{def:j_signed}
Let $\eta \in \mu_{m_i}$ be an $m_i$-th root of unity and let $f$ be a function of $\bx$. We define the
generalized sign function
\begin{align*}
    \varepsilon_{i,\eta} \bx^\lambda = \eta^{\langle\lambda, \alpha_i^\vee\rangle}\bx^\lambda.
\end{align*}
For $j\in \Z$, we define the \textit{$j$-signed part of $f$ with respect to $i$} as
\begin{align}
\label{eq:signedpart}
f_i^{(j)}(\bx) = \frac{\sum_{\eta \in \mu_{m_i}}\eta^{j-l_i-1}f(\varepsilon_{i,\eta}\bx)}{m_i}.
\end{align}
\end{definition}

\begin{remark}
If we assume that $f = \sum_\lambda a_\lambda \bx^\lambda$ has a power series expansion converging in a neighborhood of $\bx=\mathbf{0}$, then
\begin{align*}
    f_i^{(j)}(\bx) = \sum_{\delta_i(\lambda)\equiv j \pmod{m_i}}a_\lambda \bx^\lambda.
\end{align*}
\end{remark}
If $f$ is defined on a Reinhardt domain (see Definition~\ref{def:Reinhardt}) which is invariant under $\sigma_i$, the Chinta-Gunnells action becomes\footnote{
To avoid potential confusion we note that $f_i^{(j)}(\sigma_i \bx)$  is the
function $f_i^{(j)}$ evaluated at $\sigma_i \bx$.}
\be\label{eq:CG_defsigned}
f(\bx)\vert ^{\CG}_\omega \sigma_i = \frac{x_i^{l_i+1}}{1-qx_i^{m_i}}\sum_{0\leq j\leq m_i-1}f_i^{(j)}(\sigma_i \bx) \left[x_i^{-j}(1-q)-(x_i^{-m_i}-1)g_{q(\alpha_i)j} \right].
\ee

Let
\be\label{eq:DeltaD}
D(\bx;q) = \prod_{\alpha \in \Phi^+}(1-q\bx^{m_\alpha\alpha})^{\mult(\a)},\qquad
\Delta(\bx) = \prod_{\alpha \in \Phi^+}(1-\bx^{m_{\alpha}\alpha})^{\mult(\a)},
\ee
and write $D_\re$, $\Delta_\re$ for the products over the positive real roots,
and $D_\im$, $\Delta_\im$ for those over positive imaginary roots.
Recall that $\mult(\a)=1$ for $\a\in\Phi_\re$.
\begin{definition}
\label{def:cwaverage}
We define two series over the Weyl group
\begin{align*}
  Z_\omega^*(\bx;q)=\sum_{\sigma \in W}1 \vert_\omega \sigma(\bx),\qquad
  Z_\omega(\bx;q) = \frac{Z_\omega^*(\bx;q)}{\Delta(\bx)},
\end{align*}
and we refer to the latter as the \textit{twisted Chinta-Gunnells average}.
\end{definition}

We will prove that the series defining $D_\re Z_\om^*$  converges absolutely
and uniformly on compact sets in  the interior of the complexified Tits cone (Theorem~\ref{thm:convergence}).
We will also prove that
$D_\re Z_\om^*/\Delta_\re$, which agrees with $D_\re\Delta_\im Z_\om$ wherever the latter expression is defined, extends holomorphically to the same region (Proposition~\ref{prop:convergence}).
\subsection{Convergence domains}\label{sec:convdomain}
Let
\begin{align*}
    C = \{ h \in \mathfrak{h}_{\mathbb{R}}\mid \alpha_i(h) \geq 0 \text{ for }i=1,\dots,r\},
\end{align*}
be the fundamental Weyl chamber, and let
\begin{align*}
    X = \bigcup_{\sigma \in W}\sigma(C)
\end{align*}
denote the Tits cone. We also define the complexified Tits cone
$X_{\mathbb{C}} = \{x+iy \in \mathfrak{h} \mid x \in X, y \in \mathfrak{h}_\mathbb{R}\}$
and denote its interior by $X_\C^\circ$.

To discuss convergence of power series expansions in $\bx$ with respect to the regions defined above,
for $\bs\in\hi$ put
\[
s_i:=\alpha_i(\bs),\qquad q^{-\bs}:=(q^{-s_1},\ldots,q^{-s_r}).
\]
Thus, for $\lambda = \sum a_i\alpha_i \in Q$ we have $\lambda(\bs) = \sum a_i s_i$.

Recall that a series $\sum_n f_n(\bx)$ of holomorphic functions converges normally on a domain if it converges
absolutely and uniformly on every compact subset. The infinite product $\prod_n (1+f_n(\bx))$ converges normally
if $\sum_n f_n(\bx)$ converges normally. We will use repeatedly the following easy fact.
\begin{lemma}\label{lem:conv}
 Let $U$ be the interior of the set of absolute convergence for the series
 $$F(\bs)=\sum_{\lam\in Q^+} a_\lam q^{-\lam(\bs)}.$$
 Then $U$ is convex and $F$ converges normally on $U$ to a holomorphic function.
\end{lemma}
\begin{proof}
Let $F^\abs=\sum_{\lam\in Q^+} |a_\lam q^{-\lam(\bs)}|$.
 The fact that $U$  is convex follows from H\"older's inequality: if $\bs_0, \bs_1\in U$ and
$\bs_t=(1-t)\bs_0+t\bs_1$ for $0\le t\le 1$, then
\[
F^\abs(\bs_t)\le  F^\abs(\bs_0)^{1-t }F^\abs(\bs_1)^{t}<\infty.
\]
To show normal convergence, let $K_\C\subset U$  be a compact set and
$K=\{\Re(\bs)\mid \bs \in K_\C \}\subset U_\R$ its projection on the real space,
where $U_\R$ consists of real parts of elements of $U$.
We have to show that
\begin{equation}\label{eq:supsum}
\sum_{\lam  \in Q^+}\sup_{\bs \in K}|a_\lam| q^{-\lam(\bs)}
\end{equation}
converges. Take a finite cover of $K$ by closed boxes with vertices in $U_\R$.
Over each box, by convexity \eqref{eq:supsum} is bounded by the value at one of the vertices,
so on $K$ it is bounded by the sum of the values of $F^\abs$
over the finitely many vertices of the covering boxes. By the Weierstrass test,
we conclude that the $F$ converges normally to a holomorphic function.
\end{proof}
The convergence domains we will consider in this paper arise as follows.
For $R\ge 0$, let
$$C(R)=\{h\in\mathfrak h_\R\mid \a_i (h)>R\text{ for }1\leq i \leq r\},$$
so that $C(0)=C^\circ$ is the interior of the fundamental Weyl chamber.
We let
\[ X(R)=\operatorname{convex\ hull}\left(\bigcup_{\sigma\in W}\sigma C(R)\right)  \]
be the associated $W$-invariant, convex, open subset of $X^\circ$.
Let $C(R)_\C$, $X(R)_\C$ denote the corresponding tube domains, e.g. $C(R)_\C=\{x+iy\mid x\in C(R), y\in \hi_\R \}$.

To show normal convergence of a series as in the lemma on $X(R)_\C$, we will
show that the domain of absolute convergence contains $C(R)$ and it is $W$-invariant.
As an example, we determine the convergence domains of the infinite products defined above.
\begin{lemma}
\label{lem:ddeltaconvergence}
The infinite products $D_\re(q^{-\bs};q)$ and $\Delta_\re(q^{-\bs})$ converge normally on $\XC^\circ$ to holomorphic functions.
\end{lemma}
\begin{proof}
From the proof of Proposition 10.6 in \cite{Kac_1990},  we have that
\be\label{eq:Phiseries}
\sum_{\alpha \in \Phi^+}  q^{-m_\a \a(\bs)}
\ee
converges for $\bs \in C^\circ$ (as $m_\a\ge 1$). The domain of convergence is clearly
$W$-invariant as $\ss\in W$ permutes all but the finitely many terms
in the series for $\a\in \Phi(\ss)$. By Lemma~\ref{lem:conv} we conclude the series converges normally
on  $\XC^\circ$.  By the Weierstrass test, we conclude that the infinite products  $D_\re(q^{-\bs};q)$ and $\Delta_\re(q^{-\bs})$
converge normally to holomorphic functions.
\end{proof}

In general, the product $\Delta_\im(q^{-\bs})$ does not converge on the interior of the Tits cone $\XC^\circ$, due to the exponential growth of multiplicities of imaginary roots, see, for example,~\cite{CarboneEtAl2024RootMultiplicities}. As in \cite[\S 10.6]{Kac_1990}, we define
\be\label{eq:Yset}
Y= \{ \bs \in \mathfrak{h}_\R\mid
\sum_{\alpha \in \Phi^+}\mult(\alpha)q^{-m_\a \a(\bs)}< \infty \},
\ee
so that $D(q^{-\bs};q)$ and $\Delta(q^{-\bs})$ converge normally on the interior of
$Y_\C=\{x+iy\mid x\in Y, y\in \hi_\R\}$, by Lemma~\ref{lem:conv}.
\begin{remark}\label{rmk:Yaffine}
When $\Phi$ is of finite or affine type we have $Y^\circ = X^\circ$, so $\Delta(q^{-\bs})$
converges normally on $X^\circ_\C$ as before.
\end{remark}
Note that the set $Y$ is
convex and $W$ invariant (by the same argument as in Lemma~\ref{lem:conv}).
By \cite[Prop. 3.12 (c)]{Kac_1990} we have the following description of the Tits cone:
\[
X=\{h\in \hi_\R \mid \a(h)<0 \text{ only for a finite number of } \a\in\Phi^+\},
\]
which implies that $Y$ is a subset of $X$ (clearly $\a(\bs)<0$ for finitely many $\a$
in the definition of $Y$, otherwise the series does not converge).

\begin{lemma}\label{lem:setY}
$\mathrm{(a)}$ The set $Y$ contains $X(\log_q (r))$.

$\mathrm{(b)}$ The product $\Delta_\im(q^{-\bs})$  does not vanish on $Y_\C^\circ$,
and the product $D_\im(q^{-\bs})$  does not vanish on $X(R_{q,r})$, where $R_{q,r}=\max(1,\log_q r)$.
\end{lemma}
\begin{proof}
(a)  We have $\mult(\a)\le r^{\hgt(\a)}$ for $\a \in \Phi^+_\im$ by \cite[eq. (1.3.2)]{Kac_1990}.
Letting $\bs\in\hi_\R$ with $s_i\ge S>0$, the series in the definition~\eqref{eq:Yset} is bounded by:
\[\sum_{k\ge 1}r^k k^r q^{-k S}, \]
where we used that the number of $\a\in\Phi^+$ with $\hgt(\a)=k$ is at most $k^r$. This series converges
for $S > \log_q(r)$, so $C(\log_q (r))\subset Y$. Since $Y$ is convex and $W$-invariant,
it contains $X(\log_q (r))$ as well.

(b) By \cite[Proposition 5.8 (c)]{Kac_1990} it follows that
\[X^\circ\subset \{h\in \hi_\R \mid \a(h)>0 \text{ for every } \a\in \Phi^+_\im \}.
\]
Therefore for all $\bs \in Y^\circ\subset X^\circ$ and  $\a\in\Phi^+_\im$ we have
$q^{-m_\a \a (\bs)}<1$, so none of the factors in the normally convergent product $\Delta_\im$ vanishes.

By part (a), the product $D_{\text{im}}(q^{-\bs};q)$ converges normally
for $\bs\in X(\log_q r)_\C$, and we show now that none of its factors vanishes on $X(1)_\C$.
Indeed, the Weyl group permutes the imaginary roots, and for $\a\in \Phi^+_{\text{im}}$, $\ss\in W$ and
$\bs\in C(1)$ we have:
\[
\a(\ss \bs)=(\ss^{-1}\a) (\bs)>1,
\]
so $D_{\text{im}}(q^{-\bs};q)\ne 0$ for $\bs\in \ss C(1)$. Moreover, for $\bs\in X(1)$
we may write $\bs=\sum t_k \ss_k \bs_k$ with $t_k>0$,
$\sum t_k=1$ and $\bs_k\in C(1)$. By linearity it follows that $\a(\bs)>1$, so $1-q^{1-\a(\bs)}\ne 0$
for $\bs\in X(1)_\C$. We conclude that  $D_{\text{im}}(q^{-\bs};q)$ is nonzero on $X(R_{q,r})_\C$.
\end{proof}

\section{Analytic properties of the Chinta-Gunnells average}
We now show that the Weyl-group series $Z_\om^*$ has meromorphic continuation to the interior of the complexified Tits cone $ \XC^\circ$. This is known in the untwisted case, for simply-laced affine $\Phi$
and $n=2$ by~\cite[Proposition 3.1]{diaconu2023quadraticweylgroupmultiple}.

We fix throughout a twisting parameter $\omega \in P^+$
as in~\eqref{eq:twist_parameter}  and let $\rho = \sum_{i=1}^r \omega_i$ be the Weyl vector.

\subsection{Properties of the Chinta-Gunnells action}

We first prove three lemmas which will be used repeatedly. See also Lemmas 2.1, 2.9 and 2.15 in \cite{Lee2012WeylGM}.
\begin{lemma}
\label{lem:malphasign} For $\a\in\Phi$ and $1\le i\le r$ we have
$\langle m_\alpha \alpha, \alpha_i^{\vee}\rangle \equiv 0 \mod m_i$.
\end{lemma}
\begin{proof}
For $\a\in\Phi_{\text{im}}$ the statement is clear since $m_i | n=m_\a$. Assume $\a\in\Phi_{\re}$,
and recall that $\alpha^\vee = \alpha/q(\alpha)$.
We write $d_\alpha = \gcd(q(\alpha),n)$ so that $q(\alpha) = b_\alpha d_\alpha$, $n = m_\alpha d_\alpha$ and $\gcd(b_\alpha, m_\alpha) = 1$. Then
\begin{align*}
m_\alpha \langle \alpha, \alpha_i^{\vee}\rangle = \frac{b_\alpha}{b_{\alpha_i}}m_{i} \langle\alpha^\vee,\alpha_i\rangle.
\end{align*}
Since the left-hand side is an integer and $\gcd(b_{\alpha_i},m_i) = 1$ we see that $\frac{b_\alpha}{b_{\alpha_i}}\langle\alpha^\vee,\alpha_i\rangle$ is an integer, giving the desired result.
\end{proof}
Let $Q'\subset Q$ be the sublattice spanned by the set $\{ m_\alpha\alpha \}_{\a\in \Phi}$. By \cite[Lemma 2.1]{Lee2012WeylGM},
the lattice $Q'$ is invariant under $W$, and it coincides with the lattice spanned by $\{m_i \a_i\}_{i=1,\ldots, r}$.
\begin{lemma}
\label{lem:pulloutmonomial}
Let $f,g$ be functions on a $\sigma_i$-invariant Reinhardt domain such that $g(\bx)=g(\ve_{i,\eta}\bx)$
for all  $\eta \in \mu_{m_i}$. Then
    \[\left(g(\bx)f(\bx)\right)\vert_\omega^{\CG}\sigma_i = g(\ss_i \bx) \cdot f(\bx)\vert_\omega^{\CG}\sigma_i.
    \]
\end{lemma}
In particular, the lemma applies to rational functions $g(\bx)$ whose numerator and denominator are supported on $Q'$. Indeed, $\ve_{i,\eta}\bx^\lambda = \bx^\lambda$ if $\lam\in Q'$
for all $\eta \in \mu_{m_i}$ by Lemma~\ref{lem:malphasign}.
\begin{proof}
The hypothesis on $g$ gives
$(gf)_i^{(j)}(\bx)=g(\bx)f_i^{(j)}(\bx)$. Evaluating this identity at
$\sigma_i\bx$ and substituting it into~\eqref{eq:CG_defsigned} proves the claim.
\end{proof}
 For a Laurent series
$f=\sum_{\lambda \in Q}a_\lambda \bx^{\lambda}$, define $\Supp(f) = \{\lambda \mid a_\lambda \neq 0\}$.
\begin{lemma}
\label{lem:supp}
$\mathrm{(a)}$ Let $\om\in P^+$ and $\theta=\om+\rho$. We have
\begin{align*}
1\vert_\omega\sigma(\bx) = \bx^{\theta - \sigma^{-1}\theta}\frac{P_\sigma(\bx)}{\prod_{\alpha\in\Phi(\sigma)}(1-q\bx^{m_\alpha\alpha})}
\end{align*}
where $P_\sigma(\bx)$ is a polynomial in $\bx$ with $\Supp(P_\sigma) \subseteq \left\{\sum_{\alpha \in \Phi(\ss)} n_{\alpha} \alpha \mid 0 \leq n_\alpha \leq m_\alpha\right\} $.

$\mathrm{(b)}$ Let $P_\sigma(\bx) = \sum_{\lambda}a_{\ss,\lambda} \bx^\lambda$,  and let
$M = |q-1|+2 \max_t\{|g_t|\}$.  Then
\[
\sum_\lambda |a_{\ss,\lambda}| \leq M^{\ell(\ss)}.
\]
\end{lemma}
\begin{proof}
We prove the following slightly more general statement by induction on $\ell(\sigma)$: for each $\ss$,
there is a finite expression  $P_\ss(\bx)=\sum_{j} a_{\ss,j} \bx^{\lam_j} $ with $\lam_j$ in the
set above, and such that $\sum_{j} |a_{\ss,j}|\le M^{\ell(\ss)}$. Combining terms with equal exponents
gives $a_{\ss,\lam}=\sum_{\lam_j=\lam} a_{\ss,j}$ so the stated inequality follows.

The base case $\sigma = 1$ is trivial. Thus let $\ell(\sigma\sigma_i)=\ell(\sigma)+1$
and assume the statement is true for $\ss$.
By Lemma~\ref{lem:pulloutmonomial} and the induction hypothesis we have
\begin{align*}
1\vert_\omega\sigma\sigma_i  =
\frac{\big( \bx^{\theta - \sigma^{-1}\theta}P_\sigma(\bx) \big)\vert_\omega\sigma_i}{\prod_{\alpha\in\sigma_i\Phi(\sigma)}(1-q\bx^{m_\alpha\alpha})}.
\end{align*}
From the definition of the Chinta-Gunnells action, we have
\begin{align}
\label{eq:actiononpolynomial}
1\vert_\omega\sigma\sigma_i = \frac{\sum_{\lambda \in \Supp(P_\sigma)} a_{\ss,\lambda}\bx^{\sigma_i(\lam')+(l_i+1)\alpha_i}\left[x_i^{m_i-r_{m_i}(\dd_i(\lam'))}(q-1)+(1-x_i^{m_i})g_{q(\alpha_i)\delta_i(\lambda')}\right]}{\prod_{\alpha\in\Phi(\sigma\sigma_i)}(1-q\bx^{m_\alpha\alpha})},
\end{align}
where $\lam'=\theta - \sigma^{-1}\theta+\lambda$.
Note that the term in square brackets is a polynomial in $x_i$ of degree $m_i$, and that
\begin{align*}
\sigma_i(\lam')+(l_i+1)\alpha_i &=\theta-\sigma_i\sigma^{-1}\theta+\sigma_i\lambda.
\end{align*}
The induction hypothesis shows that
$\sigma_i\lambda\in \left\{\sum_{\alpha \in \ss_i\Phi(\ss)} n_{\alpha} \alpha \mid 0 \leq n_\alpha \leq m_{\sigma_i\alpha}\right\}$, and the facts that $\Phi(\sigma\sigma_i) = \{\alpha_i\}\cup\sigma_i\Phi(\sigma)$ and that $m_\alpha = m_{\sigma_i \alpha}$ finish the proof of (a).

Part (b) follows immediately from \eqref{eq:actiononpolynomial}, and the inductive hypothesis.
\end{proof}

\subsection{The domain of convergence for certain Weyl-group series}

We will compare $D_\re Z_\om^*$ to a series as in the following lemma.
\begin{lemma}
\label{lem:convergenceconstant}
Let $M \in \mathbb{R}_{>0}$ be a constant and $\theta=\om+\rho\in P^{++}$. The series
\begin{align*}
\sum_{\sigma\in W} M^{\ell(\sigma)}q^{-(\theta - \sigma^{-1}\theta)(\bs)}
\end{align*}
converges normally for  $\bs \in X^\circ_\mathbb{C}$, hence it defines a holomorphic function there.
\end{lemma}
\begin{proof}
Denote by $X_\theta\subset \hi_\R$ the domain of convergence of the series above for $\bs\in \hi_\R$.

To show $X_\theta$ contains  $ C^\circ$, let $\sigma \in W$.
Since $\theta=\om+\rho$, and $\om\ge \ss^{-1}\om$ as $\om\in P^+$ we have
\[
\theta - \sigma^{-1}\theta\ge \rho - \sigma^{-1}\rho=\sum_{\beta\in\Phi(\ss)} \beta,
\]
where the last equality is well-known and it can be proven by induction on $\ell(\sigma)$.
If $\bs \in C^\circ$, there is $S \in \mathbb{R}_{>0}$ such that $s_i > S$ for all $1\leq i\leq r$.
Therefore $\beta (\bs) >S\, \hgt(\beta)$ and we obtain
\begin{align*}
(\theta - \sigma^{-1}\theta)(\bs)\geq S \sum_{\beta\in\Phi(\ss)} \hgt (\beta).
\end{align*}
Let $k=\ell(\ss)$ and let $h_1 \leq \dots\leq h_k$ be the ordered heights of the elements $\beta\in \Phi(\ss)$.
Let $n_h$ be the number of positive roots with height at most $h$. Since the roots in $\Phi(\ss)$ are distinct we have for $1\le j\le k$:
\begin{align*}
j\leq n_{h_j}\leq (h_j + 1)^r,
\end{align*}
with the upper bound being the trivial estimate. We deduce that $h_j \ge j^{1/r}-1$, and we obtain
\begin{align}
\label{eq:heightbound}
(\theta - \sigma^{-1}\theta)(\bs)\geq S\sum_{j=1}^k(j^{1/r}-1)\ge S\left(\frac{k^{1+1/r}}{1+1/r}-k\right)
\end{align}
where the last inequality comes from comparing the sum to the integral of $x^{1/r}$.

We now have
\begin{align*}
\sum_{\sigma\in W} M^{\ell(\sigma)}q^{-(\theta - \sigma^{-1}\theta)(\bs)} \leq \sum_{k=0}^\infty M^kr^k q^{-ak^{1+1/r}+Sk}.
\end{align*}
Here the inequality comes from summing over $k=\ell(\sigma)$, applying the naive upper bound $r^k$ on the number of $\sigma \in W$ of length $k$, and applying (\ref{eq:heightbound}) with $a = \frac{S}{1+1/r}$. Absolute convergence follows by comparison to a geometric series since
$k^{1+1/r}$ grows faster than any linear function in $k$, which proves that $C^\circ\subset X_\theta$.

To show that $X_\theta$  is $W$-invariant, let $\tau \in W$ and $\bs\in X_\theta$. We have
\[
\sum_{\sigma\in W} M^{\ell(\sigma)}q^{-(\theta - \sigma^{-1}\theta)(\tau\bs)}\leq M^{\ell(\tau)}q^{(\theta-\tau^{-1}\theta)(\bs)}\sum_{\sigma' \in W}M^{\ell(\sigma')}q^{-(\theta - \sigma'^{-1}\theta)(\bs)},
\]
where we used for $\ss'=\ss\tau$ that $\ell(\sigma) \leq \ell(\sigma')+\ell(\tau)$ and
\[
(\theta - \sigma^{-1}\theta)(\tau \bs) = (\theta - \ss'^{-1}\theta)(\bs) + (\tau^{-1}\theta - \theta)(\bs)
\]
(here we assume $M\ge 1$, as for $M<1$ the statement is implied by the case $M=1$). It follows that $\tau\bs \in X_\theta$,
proving $W$-invariance.

We conclude by Lemma~\ref{lem:conv} that the series converges normally on $\XC^\circ$.
\end{proof}

\begin{theorem}
\label{thm:convergence}
Fix a twisting parameter $\omega \in P^+$. Then the series defining
$D_\re Z_\omega^*(q^{-\bs};q)$ converges normally on $\XC^\circ$.
In addition, the meromorphic function $Z_\omega^*(q^{-\bs};q)$ is $W$-invariant under the action $|_\om$
on $\XC^\circ$.
\end{theorem}

\begin{proof}
Fix a compact subset $K_\C\subset X_\C^\circ$ and let $K\subseteq X^\circ$ consist of the
real parts of $\bs\in K_\C$. By Lemma~\ref{lem:supp} we have
\begin{equation}
\label{eq:explicitactionforconvergence}
D_\re(q^{-\bs};q)(1\vert_\omega\sigma(q^{-\bs})) = q^{-(\theta - \sigma^{-1}\theta)(\bs)}P_\sigma(q^{-\bs})
\prod_{\alpha\in\Phi^+_{\text{re}} \setminus\Phi(\sigma)}(1-q^{1-m_\alpha\alpha(\bs)}).
\end{equation}
We will bound the right hand side in \eqref{eq:explicitactionforconvergence} uniformly on $K$ and then
apply Lemma~\ref{lem:convergenceconstant}.

By the convergence of~\eqref{eq:supsum},
 there exists a positive constant $B$, depending on $K$, such that
\[
\prod_{\alpha \in \Phi^+_{\text{re}}}(1+q^{1-m_\alpha \alpha(\bs)}) < B
\]
for all $\bs \in K$. In particular
\[
\prod_{\alpha\in\Phi^+_{\text{re}}\setminus\Phi(\sigma)} \big| 1-q^{1-m_\alpha\alpha(\bs)} \big| < B.
\]

Now from Lemma~\ref{lem:supp} the polynomial $P_\sigma(q^{-\bs})$ can be written as $\sum_\lambda a_{\ss,\lambda} q^{-\lambda(\bs)}$ where $\lambda = \sum_{\alpha \in \Phi(\sigma)}n_\alpha \alpha$ with $0 \leq n_\alpha \leq m_\alpha$.
If $\alpha(\bs) \geq 0$ then $q^{-n_\a \a(\bs)}\leq 1$; if $\alpha(\bs) < 0$ then
\[
q^{-n_\alpha\alpha(\bs)} \leq q^{-m_\alpha \alpha(\bs)}\leq 1 + q^{1-m_\alpha \alpha(\bs)}
\]
so that we have $q^{-\lambda(\bs)}\leq B$. Thus by Lemma~\ref{lem:supp} (b) we have
$|P_\sigma(q^{-\bs})| \leq B M^{\ell(\sigma)}$, so that
\[
|D_\re(q^{-\bs};q)(1\vert_\omega\sigma(q^{-\bs}))| \leq B^2M^{\ell(\sigma)}q^{-(\theta - \sigma^{-1}\theta)(\bs)} \qquad (\bs\in K).
\]
Lemma~\ref{lem:convergenceconstant} proves normal convergence
on $K_\C$, so the limiting function is holomorphic by the Weierstrass test.

Finally, for $\bx=q^{-\bs}$ in the region of normal convergence of $Z_\om^*$ (away from the polar hyperplanes at the zeroes of $D_\re$)
and $\ss' \in W$ we have
\[
Z_\om^*\vert_\om \ss'(\bx;q)=\sum_{\ss\in W} 1\vert_\om \ss\ss'(\bx;q)= Z_\om^*(\bx;q).
\]
\end{proof}

As in finite type, we next show that $\Delta_\re$ divides $Z_\om^*$.
\begin{proposition}\label{prop:convergence}
 Fix a twisting parameter $\omega \in P^+$. The function
 $$ D_\re(q^{-\bs};q)\frac{Z_\omega^*(q^{-\bs};q)}{\Delta_\re(q^{-\bs})}$$
 is holomorphic
 for $\bs\in \XC^\circ$. In this region,  the meromorphic function $Z_\omega^*(q^{-\bs};q)/\Delta_\re(q^{-\bs})$ is $W$-invariant under the action $\vert_\om^\CG$.
\end{proposition}
\begin{proof}
Put $\Omega(0):=\{q^{-\bs} \mid \bs \in X_\C^\circ\}$ (the
notation anticipates the sets $\Omega(R)$, $R\ge 0$, introduced in the next section), and $F(\bx)=D_\re Z_\om^*(\bx;q)$.
We first show that $F$ vanishes on
\[
H_\a:=\{\bx\in\Omega(0)\mid \bx^{m_\alpha\alpha}=1\}
\qquad (\alpha\in\Phi_\re^+).
\]
By Lemma~\ref{lem:malphasign}, $D_\re$ is invariant under the generalized sign functions.
Also, since $\ss_i$ permutes $\Phi_\re^+\setminus\{\alpha_i\}$, we have
\be\label{eq:Dre_sigmai}
\frac{D_\re(\ss_i\bx;q)}{D_\re(\bx;q)}
=\frac{1-qx_i^{-m_i}}{1-qx_i^{m_i}},
\ee
so formula~\eqref{eq:CG_defsigned} and the invariance of $Z_\om^*$ under $\vert_\om \ss_i$ gives
\begin{equation}\label{eq:Freflection}
(1-qx_i^{-m_i})F(\bx)
=x_i^{l_i+1}\sum_{j=0}^{m_i-1}F_i^{(j)}(\ss_i\bx)
\left[x_i^{m_i-j}(q-1)+(1-x_i^{m_i})g_{q(\alpha_i)j}\right].
\end{equation}
Assume $\bx\in H_{\a_i}$, that is $x_i=\zeta\in\mu_{m_i}$.
We have $\ss_i\bx=\ve_{i,\zeta^{-1}}\bx$, so Definition~\ref{def:j_signed} gives
\[
F_i^{(j)}(\ss_i\bx)=\zeta^{j-l_i-1}F_i^{(j)}(\bx).
\]
Substitution into~\eqref{eq:Freflection}, together with $\sum_j F_i^{(j)}=F$, yields
$(1-q)F(\bx)=(q-1)F(\bx)$. Since $q\ne 1$, this proves that $F$ vanishes on $H_{\alpha_i}$.

Suppose that $F$ vanishes on $H_\beta$ for $\beta\in \Phi^+$. We show that it also vanishes on  $H_{\ss_i\beta}$ whenever $\ss_i\beta\in \Phi^+$.
If $\bx\in H_{\ss_i\beta}$, then $\ss_i\bx\in H_\beta$, and Lemma~\ref{lem:malphasign}
shows that every $\ve_{i,\eta}$ preserves $H_\beta$.
Thus all the terms $F_i^{(j)}(\ss_i\bx)$ vanish, and~\eqref{eq:Freflection} gives
$(1-qx_i^{-m_i})F(\bx)=0$ on $H_{\ss_i\beta}$.
Since $\ss_i\beta\ne \alpha_i$, the first factor is nonzero on a dense subset of each component of this hypersurface,
as $x_i$ may vary on each component. By continuity, $F$ vanishes on all of $H_{\ss_i\beta}$.
Every real root is a Weyl-group translate of a simple root, so $F$ vanishes on every $H_\alpha$.

Now set $\wF(\bs)=F(q^{-\bs})$, $\wDD_\re(\bs)= \Delta_\re(q^{-\bs}) $. Both functions are holomorphic on
$X_\C^\circ$, and we now show that $\wF/\wDD_\re$ is holomorphic at a fixed $\bt\in  X_\C^\circ$.
 By the proof of Lemma~\ref{lem:ddeltaconvergence},
there are finitely many $\a\in\Phi^+_\re$ such that $q^{-m_\a \a(\bt)} =1$, and
denoting by $S_\bt$ the set of such $\a$  we can write
\[
\wDD_\re(\bs)= u(\bs) \prod_{\a\in S_\bt} (1-q^{-m_\a \a(\bs)})
\]
with $u$ holomorphic at $\bt$ and $u(\bt)\ne 0$. Let $\bs=\bt+\bh$ with $\bh$ in a neighborhood
$U$ of $\mathbf 0$ such that $u(\bt+\bh)\ne 0$ for $\bh\in U$. For $\a\in S_\bt$, we have $q^{-m_\a \a(\bt)} =1$ so
\[
1-q^{-m_\alpha\alpha(\bs)}=1-q^{-m_\alpha\alpha(\bh)}=\a(\bh) v_\a(\bh),
\]
where $v_\alpha$ is holomorphic and $v_\alpha(\mathbf 0)=m_\a \log q\ne 0$, as it follows by the Taylor expansion of the exponential.
After possibly shrinking the neighborhood $U$, we can assume all the $v_\alpha$ are nonvanishing on $U$.

We have shown above that $\wF(\bt+\bh)$ vanishes for $\bh$ in $V_\a:=\{\bh\in U: \a(\bh) =0 \}$,
so for $\a\in S_\bt$ we have
$$
\widetilde{F}(\bt+\bh)= \a(\bh) G(\bh)
$$
with $G(\bh)$ holomorphic in $U$. For $\beta\in S_\bt$, $\beta\ne \a$, it follows that
$G$ vanishes on $V_\beta\setminus V_\a$. This subset is dense in $V_\beta$ since $\beta$ and $\a$ are nonproportional, so continuity gives the vanishing of $G$ on $V_\beta$.
 Repeating the argument shows that $\wF(\bt+\bh)=g(\bh) \prod_{\a\in S_\bt} \a(\bh)$ for $\bh\in U$,
with $g$ holomorphic. Therefore $\wF/\wDD_\re (\bt+\bh)$ is holomorphic for $\bh\in U$,
proving the first claim.

The $W$-invariance follows from Theorem~\ref{thm:convergence}, together with Lemma~\ref{lem:pulloutmonomial}
and the relation
\be\label{eq:Deltassi}
\frac{\Delta_\re(\bx)}{\Delta_\re(\ss_i \bx)}=\frac{1-x_i^{m_i}}{1-x_i^{-m_i}}=-x_i^{m_i},
\ee
as $\ss_i$ permutes the roots in $\Phi_\re^+\setminus\{\a_i \}$.
\end{proof}

\begin{corollary}\label{cor:convergence} Fix a twisting parameter $\omega \in P^+$.
 The function $D_\re(q^{-\bs};q) Z_\omega(q^{-\bs};q)$ is holomorphic on $Y_\C^\circ$,
 and $Z_\omega(q^{-\bs};q)$ is $W$-invariant under the action $\vert_\om^\CG$ on $Y_\C^\circ$.
 \end{corollary}
\begin{proof}
By Lemma~\ref{lem:setY} (b), we have that $\Delta_\im^{-1}(q^{-\bs})$ is holomorphic on  $Y_\C^\circ$.
Proposition~\ref{prop:convergence} then proves the holomorphy statement.

We have that $\Delta_\im(\bx)=\Delta_\im(\ss\bx)$ for all $\ss\in W$,
as $\ss$ permutes $\Phi^+_\im$. Therefore Lemma~\ref{lem:pulloutmonomial} and Proposition~\ref{prop:convergence} prove the $W$-invariance.
\end{proof}

\begin{remark} \label{rmk:twistedchar}
Both $Z_\om$ and $Z_\om^*/\Delta_\re=\Delta_\im Z_\om$ are invariant under the action $\vert_\om^\CG$,
which illustrates the difficulty posed by imaginary roots for general $\Phi$. Although the latter function is meromorphic in a larger region,
we choose the $\pi$-part in the MDS to be $Z_\om$, as in \cite{Lee2012WeylGM}. There are two motivations
for this choice. First, in the affine case $\wD_4$,  this choice gives the canonical axiomatic MDS
after multiplying it by a certain factor supported on imaginary roots \cite{diaconu2023quadraticweylgroupmultiple}.
Second, the function  $Z_\om$ is a $q$-deformation of the Weyl-Kac character corresponding
to the module with highest weight $\om$.
Indeed, setting $n=1$ and $q=1$ in the definition of $Z_\om$ yields
a simple multiple of the character in \cite[10.6.1]{Kac_1990}.
\end{remark}

\section{Functions invariant under the Chinta-Gunnells action }\label{sec:CGinv}
We now turn our attention from the Chinta-Gunnells average itself to a more general class of invariant functions. We first introduce the regions in $\hi$ (in the $\bs$ variable)
and in $(\C^\times)^r$ (in the $\bx=q^{-\bs}$ variables) on which  such invariant functions are naturally defined as meromorphic functions.

Recall the sets $C(R)\subset C^\circ$ and $X(R)\subset X^\circ$ defined in Subsection~\ref{sec:convdomain} for $R\ge 0$.
Let
$$
 \Omega_C(R)=\{q^{-\bs} \mid \bs \in C(R)_\C \},\qquad \Omega(R)=\{q^{-\bs} \mid \bs \in X(R)_\C \}
$$
be their images in $(\C^\times)^r$. Note that $\Omega(R)$ is  $W$-invariant.

Since $X(0)_\C=X_\C^\circ$, Theorem~\ref{thm:convergence} implies that
$D_\re Z_\om^*(\bx;q)$ is holomorphic on $\Omega(0)$. By Lemma~\ref{lem:setY} (a),
$X(\log_q r)_\C\subset Y_\C^\circ $; hence
Corollary~\ref{cor:convergence} shows that $D_\re Z_\om(\bx;q)$ is holomorphic on
$\Omega(\log_q r)$.

We will consider more generally functions on $\Omega(R)$ satisfying the following assumptions.

\begin{hypothesisH}
\label{hyp:analyticassumptions}
Let $R\ge 0$.  A meromorphic function $Z$ on $\Omega(R)$ satisfies Hypothesis $H(R)$ if:
\begin{itemize}
    \item The function $D_\re(\bx;q) Z (\bx)$ is holomorphic on $\Omega(R)$;
    \item The function $Z(\bx)$ has a Taylor expansion:
\[
Z(\bx)=\sum_{\lam\in Q^+}a_\lam(Z) \bx^{\lam},
\]
which converges absolutely for $\bx\in \Omega_C(R')$, for $R'=\max\{R,m_1^{-1},\ldots,m_r^{-1}\}$.
\end{itemize}
\end{hypothesisH}
The infinite product $D_\re(\bx;q)$ is nonzero for $\bx \in \Omega_C(\max_i\{m_i^{-1}\})$, so the first assumption
implies that  $Z$ is holomorphic on $\Omega_C(R')$, hence it equals its Laurent expansion
at the origin on $\Omega_C(R')$ by Proposition~\ref{prop:Reinhardt} below. The
second assumption is equivalent to $Z$ extending holomorphically
to the polydisk
$$P(R'):=\{\bx\in\C^r \mid |x_i|<q^{-R'}\}$$
from its initial domain $\Omega_C(R')=P(R')\cap (\C^\times)^r$. [Proof:
It is enough to show that the series converges normally
on $P(R')$. For a compact $K\subset P(R')$, assume $\sup_K |x_i|<t_i<q^{-R'}$;
by absolute convergence at $\bt=(t_1, \ldots, t_r)\in \Omega_C(R')$,
the series $\sum_{\lam\in Q^+} |a_\lam(Z)| \bt^\lam$ converges, and for every $\bx\in K$
we have $|\a_\lam(Z) \bx^\lam|\le |\a_\lam(Z) | \bt^\lam$, so the series converges normally
on $P(R')$ by the Weierstrass test.] The second assumption is clearly satisfied
by the functions $Z_\om^*$ and $Z_\om$, as well as by the MDS considered in the next section.

We recall the following definition \cite[Definition 2.4.4]{Hrmander1973AnIT}.
\begin{definition}
\label{def:Reinhardt}
An open set $\Omega \subseteq \mathbb{C}^r$ is called a \textit{Reinhardt domain} if $(z_1, \ldots, z_r) \in \Omega$ implies $(e^{i t_1}z_1, \ldots, e^{i t_r}z_r) \in \Omega$ for all real $t_1, \ldots, t_r$.
\end{definition}

We will use the following property of holomorphic functions on connected Reinhardt domains.
\begin{proposition}[\cite{JarnickiPflug2008} Prop. 1.7.15]
\label{prop:Reinhardt}
A holomorphic function on a connected Reinhardt domain $\Omega$
is equal to its Laurent expansion at $\mathbf{0}$ on the whole of $\Omega$.
\end{proposition}

The Reinhardt domains we will consider are obtained as follows. Fix $1\le i\le r$ and let
\[
A_i(R)=\operatorname{conv\ hull}\bigl(C(R)\cup\sigma_i C(R)\bigr).
\]
The tube $A_i(R)_\C$ is contained in $X(R)_\C$ and is invariant under $\sigma_i$.  Its image
\be\label{eq:Omegai}
\Omega_i(R)=\{q^{-\bs}\mid \bs\in A_i(R)_\C\}\subset \Omega(R)
\ee
is a connected Reinhardt domain, invariant under $\ss_i$.

\begin{lemma}\label{lem:Di}
Fix $1\le i\le r$, and $R\ge \max_j \{m_j^{-1}\}$.  The product
\[
D_{\re,i}(\bx;q):=\prod_{\alpha \in \Phi_{\re}^+\setminus\{\alpha_i\}}(1-q\bx^{m_\alpha\alpha})
\]
defines a holomorphic function on $\Omega_i(R)$ invariant under $\ss_i$,
which does not vanish on $\Omega_i(R)$.
\end{lemma}
\begin{proof}
By Lemma~\ref{lem:ddeltaconvergence}, the product $D_{\re,i}(\bx;q)$ converges normally on
$\Omega_i(R)$, as $A_i(R)_\C\subset X_\C^\circ$. Since $\ss_i$ permutes the
roots in $\Phi^+_\re\setminus\{\alpha_i\}$, it follows that $D_{\re,i}(\bx;q)=D_{\re,i}(\ss_i\bx;q)$.
Moreover it follows first on $C(R)\cup\sigma_i C(R)$, and then by convexity, that
\[
\Re\bigl(\beta(\bs)\bigr)>R
\qquad(\bs\in A_i(R)_\C)
\]
for all $\beta \in \Phi^+_\re\setminus\{\alpha_i\}$.
In particular, $|q\bx^{m_\beta\beta}|=q^{1-m_\beta\Re(\beta(\bs))}<1$
on $\Omega_i(R)$, as $R\ge 1/m_\beta$. Since the product converges normally and none
of its factors vanishes, it follows that $D_{\re,i}(\bx;q)$
does not vanish on $\Omega_i(R)$.
\end{proof}

\subsection{Recursion relations for coefficients of invariant functions}
We now show that the invariance of a function $Z$
under either the action $\vert_\omega^{\CG}$, or the action $\vert_\omega$ is encoded by certain recursion relations between the coefficients of its Taylor series expansion at~$\mathbf{0}$, if $Z$
satisfies Hypothesis~\ref{hyp:analyticassumptions} for some $R$.

\begin{lemma}
\label{lem:iparteq}
Let $F(\bx)$ be a meromorphic function on a $\sigma_i$-invariant Reinhardt domain $Y \subset (\mathbb{C}^\times)^r$. Then
\begin{align*}
(F(\bx)\vert_\omega^{\CG}\sigma_i)^{(j)}_i = \frac{x_i^{l_i+1}}{1-qx_i^{m_i}}\left[F_i^{(j)}(\sigma_i\bx)x_i^{-j}(1-q)-F_i^{(-j)}(\sigma_i\bx)(x_i^{-m_i}-1)g_{-jq(\alpha_i)}\right].
\end{align*}
In particular if $j=0$ then
\begin{align*}
(F(\bx)\vert_\omega^{\CG}\sigma_i)^{(0)}_i = x_i^{l_i+1-m_i}F_i^{(0)}(\sigma_i\bx).
\end{align*}
The same formulas hold for the $\vert_\omega$ action, after multiplying the right-hand sides by $-x_i^{m_i}$.
\end{lemma}
\begin{proof}
From the definition of the Chinta-Gunnells action in terms of signed parts we have that
\begin{align}
\label{eq:dzsign}
F^{(j)}_i(\bx)\vert ^{\CG}_\omega \sigma_i = \frac{x_i^{l_i+1}}{1-qx_i^{m_i}}F_i^{(j)}(\sigma_i \bx)x_i^{-j}(1-q)-\frac{x_i^{l_i+1}}{1-qx_i^{m_i}}F_i^{(j)}(\sigma_i \bx)(x_i^{-m_i}-1)g_{q(\alpha_i)j}.
\end{align}

Using Definition~\ref{def:j_signed}, we see that the first term in the sum is its own $j$-signed part with respect to $i$, while the second term is its own $(-j)$-signed part.
The first relation then follows from the uniqueness of the decomposition $F(\bx) = \sum_{0\leq j \leq m_i-1}F_i^{(j)}(\bx)$. This specializes to the second relation when $j=0$.
\end{proof}

\begin{proposition}
\label{prop:relations}
Fix a twisting parameter $\omega \in P^+$ and let $R\ge 0$.
Let $Z(\bx)$ be a meromorphic function on $\Omega(R)$ satisfying
Hypothesis~\ref{hyp:analyticassumptions} and invariant under the
$W$-action $\vert_\omega^{\CG}$.

Then for $i=1,\ldots,r$ and $j=r_{m_i}(\delta_i(\lambda))$,  its Taylor coefficients
$a_\lam=a_\lam(Z)$ satisfy
\[
a_\lam=\begin{cases}
    a_{\sigma_i\lambda +(l_i+1-m_i)\alpha_i} & \text{ if } j=0\\[5pt]
    g_{-jq(\alpha_i)}a_{\lambda + (j-m_i)\alpha_i} + a_{\sigma_i\lambda+(l_i+1-j)\alpha_i} - g_{-jq(\alpha_i)}a_{\sigma_i\lambda+(l_i+1-m_i)\alpha_i} & \text{ if } j\ne 0,
\end{cases}
\]
where we set $a_\lam=0$ for $\lam\notin Q^+$.
\end{proposition}
\begin{proof}
Fix $i$, and let $\Omega_i(R)$ be the
connected Reinhardt domain defined in~\eqref{eq:Omegai}.
By assumption, $D_\re Z(\bx)$ is holomorphic on $\Omega_i(R)$ and by  Lemma~\ref{lem:pulloutmonomial}
we have that
\begin{align*}
(D_\re Z)^{(j)}_i = D_\re Z^{(j)}_i
\end{align*}
is also holomorphic on $\Omega_i(R)$.
By Lemma~\ref{lem:ddeltaconvergence}, the product $D_\re(\bx;q)$ converges absolutely
in the $\ss_i$-invariant region $\Omega_i(R)$, where it satisfies~\eqref{eq:Dre_sigmai}.

\textbf{Assume $j=0$}.  Applying Lemma~\ref{lem:iparteq} to $Z$, multiplying both sides
by $D_\re$ and using the above facts we obtain the identity of  holomorphic functions on $\Omega_i(R)$
\begin{align}
\label{eq:untwisted0powerseriesrelation}
D_\re Z_i^{(0)}(\bx) = x_i^{l_i+1-m_i}\frac{1-q\bx^{m_i\alpha_i}}{1-q\bx^{-m_i\alpha_i}}D_\re Z_i^{(0)}(\sigma_i\bx).
\end{align}
Since  $\Omega_i(R)$ is $\ss_i$-invariant,
 $D_\re Z_i^{(0)}(\sigma_i\bx)$ is also holomorphic on it, so
 the right-hand side vanishes at $q\bx^{m_i\alpha_i}=1$.
 After dividing by $1-qx_i^{m_i}$, we conclude that
\begin{align*}
\prod_{\alpha \in \Phi_{\re}^+\setminus\{\alpha_i\}}(1-q\bx^{m_\alpha\alpha})Z_i^{(0)}(\bx)
\end{align*}
is holomorphic on $ \Omega_i(R)$.

Let $R'=\max\{R,m_1^{-1},\dots,m_r^{-1}\}$, and restrict to $\Omega_i(R')\subset \Omega_i(R)$.
The product $D_{\re,i}$ above does not vanish on $\Omega_i(R')$
by Lemma~\ref{lem:Di} and is invariant under $\ss_i$.
Therefore $Z_i^{(0)}$ is holomorphic on the Reinhardt domain $\Omega_i(R')$,
where it satisfies
 \be\label{eq:Zi0}
 Z_i^{(0)}(\bx) = x_i^{l_i+1-m_i}Z_i^{(0)}(\sigma_i\bx).
 \ee
 By Hypothesis~\ref{hyp:analyticassumptions}, we have
 $Z_i^{(0)}(\bx)=\sum_{\lam \in Q^+}a_\lam \bx^\lam$ on $\Omega_C(R')\subset \Omega_i(R')$, and
 by Proposition~\ref{prop:Reinhardt}, this Taylor expansion holds
 throughout the set $\Omega_i(R')$. Since this set is invariant under $\ss_i$,
we conclude that for $\bx\in \Omega_i(R')$ the function $Z_i^{(0)}(\sigma_i\bx)$ has Laurent expansion
 \[
 Z_i^{(0)}(\sigma_i\bx)=
 \sum_{\substack{\lambda\in Q^+\\r_{m_i}(\delta_i(\lambda))=0}}a_\lambda \bx^{\sigma_i\lambda},
 \]
 and equating coefficients in~\eqref{eq:Zi0} gives the desired relation.

\textbf{Assume now $0< j < m_i$}. We define
\begin{align}
\label{eq:Hdefinition}
H_i^{(j)}(\bx) &= Z_i^{(j)}(\bx)-x_i^{m_i-j}g_{-jq(\alpha_i)}Z^{(-j)}_i(\bx);
\end{align}
note as before that $D_\re H_i^{(j)}(\bx)$ is holomorphic on $\Omega_i(R)$.
We apply Lemma~\ref{lem:iparteq} to $j$ and $-j$ to obtain the identity of functions
\begin{align*}
D_\re H_i^{(j)}(\bx)= x_i^{l_i+1-j} \frac{1-q\bx^{m_i\alpha_i}}{1-q\bx^{-m_i\alpha_i}} D_\re H_i^{(j)}(\sigma_i\bx),
\end{align*}
and the same argument as for $j=0$ shows that $H_i^{(j)}(\bx)$ is holomorphic on $\Omega_i(R')$  and that
\[
H_i^{(j)}(\bx)=x_i^{l_i+1-j}  H_i^{(j)}(\sigma_i\bx).
\]
By assumption, $H_i^{(j)}$ equals its Taylor series
$H_i^{(j)}(\bx)=\sum_{\lambda \in Q^+}h_\lambda \bx^\lambda$ for $\bx\in \Omega_C(R')$,
and by Proposition~\ref{prop:Reinhardt} this holds throughout $\Omega_i(R')$.  As before we obtain
$h_\lam=h_{\ss_i\lam+(l_i+1-j)\a_i}$, and from~\eqref{eq:Hdefinition} we have
$h_\lam=a_\lam-g_{-jq(\alpha_i)}a_{\lam+(j-m_i)\a_i}$, finishing the proof of the second relation.
\end{proof}

\begin{corollary}
\label{cor:coeffsdefine}
In the notation of Proposition~\ref{prop:relations} the coefficients $a_\lambda$ are determined by those with $\langle \lambda, \alpha_i^\vee\rangle \leq l_i$ for all $i$.
\end{corollary}
\begin{proof}
Let $\langle \lambda, \alpha_i^\vee\rangle > l_i$. The result follows from the fact that if $a_{\lambda'}$ is a coefficient appearing on the right-hand side of the relations in Proposition~\ref{prop:relations} then $\lambda' < \lambda$. Since all such $\lambda'$ are of the form $\lambda + k\alpha_i$ it is enough to check that each $k<0$. This can be seen directly.
\end{proof}

We have a similar result for a function invariant under the $\vert_\omega$ action.
\begin{proposition}
\label{prop:starrelations}
Fix a twisting parameter $\omega \in P^+$ and let $R\ge 0$.
Let $Z(\bx)$ be a meromorphic function on $\Omega(R)$ satisfying
Hypothesis~\ref{hyp:analyticassumptions} and invariant under the
$W$-action $\vert_\omega$.

For $i=1,\ldots,r$ and $j=r_{m_i}(\delta_i(\lambda))$, the Taylor coefficients $c_\lam=a_\lam(Z)$ satisfy
\[
c_\lam=\begin{cases}
    -c_{\sigma_i\lambda +(l_i+1)\alpha_i} & \text{ if } j=0\\[5pt]
    g_{-jq(\alpha_i)}c_{\lambda + (j-m_i)\alpha_i} - c_{\sigma_i\lambda+(l_i+1+m_i-j)\alpha_i} + g_{-jq(\alpha_i)}c_{\sigma_i\lambda+(l_i+1)\alpha_i}  & \text{ if } j\ne 0,
\end{cases}
\]
where we set $c_\lam=0$ for $\lam\notin Q^+$.
\end{proposition}
\begin{proof}
First assume $j=0$. From Lemmas~\ref{lem:iparteq} and \ref{lem:pulloutmonomial}, we have the identity of functions
\begin{align*}
D_\re Z_i^{(0)}(\bx) = -x_i^{l_i+1}\frac{1-q\bx^{m_i\alpha_i}}{1-q\bx^{-m_i\alpha_i}}D_\re Z_i^{(0)}(\sigma_i\bx)
\end{align*}
on $\Omega_i(R)$. The rest of the proof for $j=0$ follows exactly as in Proposition~\ref{prop:relations}.

If $0 < j < m_i$ we define
$H^{(j)}_i(\bx) = Z_i^{(j)}(\bx)-x_i^{m_i-j}g_{-jq(\alpha_i)}Z^{(-j)}_i(\bx).$
From Lemmas~\ref{lem:iparteq} and \ref{lem:pulloutmonomial} we obtain an identity of functions
\begin{align*}
D_\re H^{(j)}_i(\bx)=- x_i^{l_i+1+m_i-j} \frac{1-q\bx^{m_i\alpha_i}}{1-q\bx^{-m_i\alpha_i}} D_\re H^{(j)}_i(\sigma_i\bx)
\end{align*}
on $\Omega_i(R)$. Again, the rest of the proof follows exactly as in Proposition~\ref{prop:relations}.
\end{proof}

\begin{corollary}
\label{cor:starcoeffsdefine}
In the notation of Proposition~\ref{prop:starrelations}, the coefficients $c_\lambda$ are determined by those with $\langle \lambda, \alpha_i^\vee\rangle \leq l_i$ for all $i$.
\end{corollary}
\begin{proof}

Suppose $\langle \lambda, \alpha_i^{\vee}\rangle > l_i$. As before, all $\lambda'$
on the right-hand side of the relations in Proposition~\ref{prop:starrelations}
are of the form $\lambda + k\alpha_i$ and it is enough to check that each such $k<0$.

Assume that $\delta_i(\lambda) \equiv 0 \pmod {m_i}$. If $\langle \lambda, \alpha_i^{\vee}\rangle > l_i+1$, then the result is clear. If $\langle \lambda, \alpha_i^{\vee}\rangle = l_i+1$, then the relation becomes $c_\lambda = -c_\lambda$ and in fact $c_\lambda = 0$.

Assume that $\delta_i(\lambda) \equiv j \pmod {m_i}$ for $1\leq j\leq m_i-1$. By assumption we have $\langle \lambda, \alpha_i^\vee\rangle \equiv l_i + 1-j \pmod {m_i}$. From the assumption $\langle \lambda, \alpha_i^\vee\rangle > l_i$, we in fact have $\langle \lambda, \alpha_i^\vee \rangle \geq l_i + 1 +m_i - j$. When the inequality is strict the result is clear, and when $\langle \lambda, \alpha_i^\vee \rangle = l_i + 1 +m_i - j$ the relation becomes $c_\lambda = g_{-jq(\alpha_i)}c_{\lambda + (j-m_i)\alpha_i}$.
\end{proof}

\subsection{Functional equations in polynomial form}\label{sec:local_poly}

For applications, it is convenient to restate the relations in the preceding subsection
in terms of certain polynomials of one variable defined as follows.
Let $Z$ be a function invariant under the action $\vert_\om^\CG$ as in Proposition~\ref{prop:relations}.
Decompose its $j$-signed part as
\begin{equation}\label{eq:local-signed-polynomials}
 Z_i^{(j)}(\bx)
 =\sum_{\wlam=\sum_{k\ne i} n_k\a_k\in Q^+}
   \bx^{\widehat\lambda}
   P_{\wlam}^{(j)}(x_i),
\end{equation}
where $P_{\wlam}^{(j)}$ is a one-variable power series. Defining for $0\le j<m_i$
\begin{equation}\label{eq:signed-polynomial-combination}
 \Pc_{\wlam}^{(j)}(x)=
 \begin{cases}
 P_{\wlam}^{(j)}(x)
 -g_{-q(\a_i)j}x^{m_i-j}
  P_{\wlam}^{(-j)}(x),&j\ne0,\\
 P_{\wlam}^{(0)}(x),&j=0,
 \end{cases}
\end{equation}
it follows from Proposition~\ref{prop:relations} that the power series $\Pc_{\wlam}^{(j)}(x)$
satisfy the functional equations ($l_i=\la\om,\a_i^\vee\ra$)
\begin{equation}\label{eq:local-polynomial-FE}
 \Pc_{\wlam}^{(j)}(x)= \Pc_{\wlam}^{(j)}(x^{-1})\cdot
 \begin{cases}
 x^{l_i+1-\la\wlam,\a_i^\vee\ra-j },
       &0<j<m_i,\\
 x^{l_i+1-\la\wlam,\a_i^\vee\ra-m_i},
       &j=0,
 \end{cases}
\end{equation}
which in particular imply that they are \emph{polynomials}.
Note that $$(1-qx^{m_i}) P_{\wlam}^{(j)}(x)=\Pc_{\wlam}^{(j)}(x)
 +g_{-q(\a_i)j}x^{m_i-j} \Pc_{\wlam}^{(-j)}(x)$$
 are also polynomials for $0<j<m_i$.

\subsection{Linear independence}
\label{subsec:linear}

Let $\om\in P^+$ with $\la \om, \a_i^\vee \ra = l_i \in \mathbb{Z}_{\ge 0}$.
As in \cite{Chinta2006OnTP}, one can characterize as follows the elements $\lam\in Q^+$ with
$$\langle \lambda, \alpha_i^\vee\rangle \leq l_i\qquad  (i=1,\ldots, r), $$
appearing in Corollaries \ref{cor:coeffsdefine} and \ref{cor:starcoeffsdefine}. Setting $\xi=\om-\lam$,
we have $\xi \leq \omega$ by the definition of the usual partial ordering on $P$.
Furthermore $\langle \xi,\alpha_i^\vee\rangle = l_i - \langle \lambda, \alpha_i^\vee\rangle \geq 0$, so $\xi\in P^+$.
Thus if we let
$$\Pi_{\omega} = \{\xi \in P^+ \mid \xi \leq \omega\}$$
as in~\eqref{eq:Pi_omega},
we conclude that an invariant function $Z$ as in Propositions~\ref{prop:relations} or~\ref{prop:starrelations} is completely
determined by its Taylor coefficients $a_{\om-\xi}$ for $\xi\in\Pi_\om$. Note that $\Pi_\om$ has a representation-theoretic interpretation,
containing the set of dominant weights in the representation with highest weight $\omega$
(the latter set consists of the nondegenerate elements  $\xi\in\Pi_\om$, namely for which  $\la\om, \a_i^\vee\ra>0$ for at least one $i$
in each connected component of the support of $\om-\xi$ \cite[Prop. 11.2]{Kac_1990}).
For finite root systems,
a similar result for the coefficients of a function $DZ$ with $Z$ invariant under the $|_\om^{\CG}$ action was proved in  \cite{Chinta2006OnTP} in the quadratic case, and for arbitrary $n$ in  \cite{friedlanderpparts}.

As in the finite case we will show that, for any $\xi\in \Pi_\omega$, the function
$\bx^{\omega-\xi}Z_\xi^*$
is invariant under the action $|_\om$ and has Taylor coefficients
$a_{\omega - \xi} = 1$ and $a_{\omega - \xi'} = 0$ for all $\xi' \in \Pi_\omega$ with $\xi \neq \xi'$.
We will therefore show that the set $\{\bx^{\omega-\xi}Z_\xi^*\}_{\xi\in\Pi_\om}$ plays the role
of a basis for the space of functions $Z$ as in Proposition~\ref{prop:starrelations}.

Recall the Weyl group invariant domains $\Omega(R)$ defined in the beginning of this section.
\begin{lemma}
\label{lem:cginvarianceofbasis}
The function $\bx^{\omega - \xi}Z_\xi^*(\bx;q)$ is  $W$-invariant under the action
$\vert_\omega$ on the set $\Omega(0)$, and
the function $\bx^{\omega - \xi}Z_\xi(\bx;q)$ is  $W$-invariant under the action
$\vert_\omega^{\CG}$ on the set $\Omega(\log_q r)$.
\end{lemma}
\begin{proof}
One may check explicitly that
\begin{align*}
\bx^{\omega-\xi}\cdot1\vert_{\xi} \sigma = \bx^{\omega-\xi}\vert_\omega \sigma.
\end{align*}
Summing both sides over $\sigma \in W$ and using the normal convergence
of $D_\re Z_\xi^*$ on compact subsets of $\Omega(0)$ (Theorem~\ref{thm:convergence})  gives the first statement.
By~\eqref{eq:Deltassi} and Lemma~\ref{lem:pulloutmonomial}, the second statement follows,
noting that we can divide by $\Delta_\im(\bx)$ on  $\Omega(\log_q r)$ by
Lemma~\ref{lem:setY}.
\end{proof}
For a power series $f=\sum a_\lam \bx^{\lam}$,
recall that  $\Supp (f)=\{\lam \mid a_\lam\ne 0\}$ denotes its support.
\begin{lemma}
\label{lem:disjointsups}
Let $\xi, \xi' \in \Pi_\om$. Then
\begin{align*}
\omega - \xi' \in \Supp (\bx^{\omega-\xi}Z^*_{\xi})
\end{align*}
if and only if $\xi = \xi'$.
\end{lemma}
\begin{proof}
One direction is clear from the fact that $Z_\xi^*$ has constant term $1$.

For the reverse direction, we assume that $\omega - \xi' \in \Supp (\bx^{\omega-\xi}Z_{\xi}^*)$, so that there exists $\sigma \in W$ with $\omega-\xi' \in \Supp (\bx^{\omega-\xi}\cdot1\vert_{\xi} \sigma)$. By Lemma~\ref{lem:supp} we have
\begin{align}
\label{eq:disjointsupscalc}
\omega-\xi' = \omega + \rho - \sigma^{-1}(\xi + \rho) + \sum_{\alpha \in \Phi(\sigma)}n_\alpha \alpha,
\end{align}
with $n_\a\ge 0$.
It follows from (\ref{eq:disjointsupscalc}) that $\sigma^{-1}(\xi) - \xi' = \gamma$ with $\gamma \in Q^+$ and $\sigma(\gamma)\le 0$.

Now by definition $\sigma^{-1}(\xi) \geq \xi'$. But $\xi$ is dominant and therefore is the largest weight in its Weyl orbit. Thus $\xi \geq \xi'$.

Similarly $\xi -\sigma(\xi') = \sigma (\gamma)\le 0$. The same argument gives $\xi' \geq \xi$ and thus $\xi = \xi'$.
\end{proof}

\begin{remark}
The function $Z_\omega^*$ is completely determined by the recursion of Proposition~\ref{prop:starrelations}, and the initial conditions $c_0=1$, $c_{\om-\xi}=0$ for $\xi\in \Pi_\om \setminus \{\omega\}$,
which follow from Lemma~\ref{lem:disjointsups}.
\end{remark}
We are now ready to prove the main result of this section.
\begin{theorem}
\label{the:basis}
Fix a twisting parameter $\omega \in P^+$ and let $R\ge \log_q r$.
Let $Z(\bx)$ be a meromorphic function on $\Omega(R)$ satisfying
Hypothesis~\ref{hyp:analyticassumptions} and invariant under the
$W$-action $\vert_\omega^{\CG}$.

Let $R'=\max\{R,m_1^{-1},\ldots,m_r^{-1}\}$. Then we have a unique decomposition
\begin{align*}
Z(\bx)=\sum_{\xi \in \Pi_\omega}c_{\om-\xi} \bx^{\omega - \xi}Z_\xi(\bx)
\end{align*}
which converges normally in the polydisk $P(R')=\{\bx\in \C^r \mid |x_i|<q^{-R'}\}$.
Moreover, $c_\lam$ is the coefficient of $\bx^{\lam}$
in the Taylor series expansion of $\Delta Z$ at $\mathbf{0}$.
\end{theorem}
\begin{proof}
Since $Z$ is a $\vert^{\CG}_\omega$-invariant function we have that $\Delta Z$ is a $\vert_\omega$-invariant function by Lemma~\ref{lem:pulloutmonomial} and~\eqref{eq:Deltassi}.
Because $R\ge\log_q r$, Lemmas~\ref{lem:ddeltaconvergence} and~\ref{lem:setY}(a)
show that $\Delta(\bx)$ is holomorphic on $\Omega(R)$ and a simple application of the maximum-modulus
principle (as in the last paragraph of the proof) shows that
its product expansion converges normally on compact subsets of $P(R')$.

Write $\Delta Z= \sum_\lambda c_\lambda \bx^\lambda$ for its Taylor series expansion, which converges
absolutely on $P(R')$ by Hypothesis~\ref{hyp:analyticassumptions}.
We note from Lemma~\ref{lem:supp} that the formal series
\begin{align*}
  \sum_{\xi \in \Pi_\omega}c_{\omega - \xi} \bx^{\omega - \xi}Z^*_\xi=\sum_{\lam\in Q^+} c_\lam' \bx^\lam
\end{align*}
has coefficients $c_\lam'$ that depend on only finitely many terms in the sum.
Since the coefficients of each series $\bx^{\omega - \xi}Z^*_\xi$ satisfy the relations in Proposition~\ref{prop:starrelations} for twisting parameter $\om$,
it follows that the coefficients $c_\lam'$ satisfy the same relations as those of $\Delta Z$. By Lemma~\ref{lem:disjointsups},
$c_{\om-\xi}'=c_{\om - \xi}$ for all $\xi\in\Pi_\om$, so the series above coincides with the Taylor series of $\Delta Z$,
by Corollary~\ref{cor:starcoeffsdefine}.
Dividing both sides by $\Delta(\bx)$--which is possible for $\bx=q^{-\bs}$ with $\bs\in C(R)_\C$
as $R\ge \log_q r$, by Corollary~\ref{cor:convergence} and Lemma~\ref{lem:setY}-- gives the desired expansion. Uniqueness follows immediately from Lemma~\ref{lem:disjointsups}.

We now show that the right-hand side of
\be\label{eq:Dre-starred-expansion}
D_{\re}(\bx;q) \Delta Z(\bx)
 =\sum_{\xi\in\Pi_\omega}
  c_{\omega-\xi}\bx^{\omega-\xi}D_{\re}(\bx;q)Z_\xi^*(\bx;q)
\ee
converges absolutely and uniformly on compact subsets $K\subset \Omega_C(R')$. Note that the bounds on the polynomials $P_\ss$ in Lemma~\ref{lem:supp} are independent of $\om$, as are the estimates in the proof of Theorem~\ref{thm:convergence} involving $\theta-\ss^{-1}\theta\ge \rho-\ss^{-1}\rho$. It follows that
there is a constant $M_K$ such that
\[
 \sup_{\xi\in\Pi_\omega}\sup_{\bx\in K}
 |D_{\re}(\bx;q)Z_\xi^*(\bx;q)|\leq M_K.
\]
On the other hand,
\[
 \sum_{\xi\in\Pi_\omega}
 |c_{\omega-\xi}\bx^{\omega-\xi}|
\]
is a subseries of the absolutely convergent Taylor series of $\Delta Z$.
Consequently the series on the right of
\eqref{eq:Dre-starred-expansion} converges normally on $\Omega_C(R')$.
Since $D_{\re}(\bx;q) \Delta (\bx)$ does not vanish on $\Omega_C(R')$, the right-hand side of the decomposition in the theorem also converges normally there.

To prove normal convergence on $P(R')$, let $K\subset P(R')$ compact
and choose radii $\rho_i<q^{-R'}$ such that $K\subset \prod_i\{|x_i|\leq\rho_i\}$.
On the boundary $T=\prod_i\{|x_i|=\rho_i\}\subset\Omega_C(R')$, the preceding estimate
uniformly bounds $D_{\re}Z_\xi^*$. These functions extend holomorphically
to the closed polydisk, so the maximum-modulus principle gives the same
bound inside. The subseries
$\sum_{\xi\in\Pi_\omega}|c_{\omega-\xi}\bx^{\omega-\xi}|$
converges uniformly there because it is a subseries of the absolutely
convergent Taylor series of $\Delta Z$. This proves normal convergence on
$P(R')$. Since $D_{\re}\Delta$ is holomorphic and nonzero there, division
yields the theorem's asserted convergence.
\end{proof}
Next, we restate the theorem above in two important cases.

\subsection{Example: the untwisted case}\label{sec:untwisted}

Assume $\om=0$ in this subsection.

If $\Phi$ is finite, then $\Pi_0=\{ 0\}$, and we recover the
following slight generalization of the result in \cite{Chinta2006OnTP, Friedlander2018TwistedWG}:
if $Z$ is invariant under $|_0^\CG$ with $DZ$ holomorphic on $\C^r$, and $Z(0,\ldots, 0)=1$, then
$Z=Z_0$ (in particular $DZ$ is actually a polynomial).

In the general case, the elements $\lam\in Q^+$ such that $\la \lam,\a_i^\vee\ra\le 0$, $i=1,\ldots, r$,
in Corollaries~\ref{cor:coeffsdefine} and~\ref{cor:starcoeffsdefine} have an interpretation in terms of imaginary
roots in $\Phi$. Namely, it is shown in~\cite[Sec. 5.3]{Kac_1990} that the set
\[
K:=\{\lam\in Q^+\setminus \{0\}\mid \la\lam, \a_i^\vee \ra \le 0  \text{ for all } i,\ \Supp(\lam)  \ \text{connected}  \}
\]
consists of imaginary roots, and we have $\Phi_{\text{im}}^+ = \bigcup_{\ss\in W} \ss(K)$.

For irreducible affine root systems, we have $K=\{n\dd \mid n\in \Z_{\ge 1}\}$ for $\dd$ the minimal positive imaginary root.
We recover a result proved in~\cite[Thm. 3.7]{BucurDiaconu2010Moments} for $\wD_4$.

For the star-shaped root system $\Phi_r$ mentioned in the introduction with $r\ge 4$,
the connectedness condition is satisfied automatically by any $\lam \in Q^+\setminus \{0\}$
with $\la\lam, \a_i^\vee \ra\le 0$ for all $i$,  so an invariant power series
$\sum a_\lam \bx^\lam $ satisfying the conditions of Proposition~\ref{prop:relations} is completely
determined by $a_0=1$ and $a_\lam$ for $\lam\in K$.

\subsection{Example: the affine case}

Assume now that $\Phi$ is an irreducible affine root system with minimal positive imaginary root $\dd\in\Phi^+_{\text{im}}$,
and consider a twisting parameter $\omega \in P^+$.

For a weight $\mu\in \Pi_\om$, we call $\mu$ \emph{maximal} if $\mu+\dd\notin \Pi_\om$.
Then the set of maximal weights
\[ \Pi_\om^{\max }:=\{ \mu \in \Pi_\om \mid \mu \text{ maximal}  \}
\]
is finite and $\Pi_\om$ is the disjoint union of the sets $\{\mu-n\dd\mid n\in \Z_{\ge 0} \}$
for $\mu\in\Pi_\om^{\max }$~\cite[Sec. 20.3]{Carter2005}. Note that $\la\dd, \a_i^\vee\ra=0$ for all $i$,
so the Chinta-Gunnells action twisted by $\mu-n\dd$ is the same as the action twisted by $\mu$.
Therefore Theorem~\ref{the:basis} can be restated as follows in terms of a finite decomposition,
which is closer in spirit to the result for finite $\Phi$.

\begin{theorem}
\label{the:decompaaffine} Assume $\Phi$ is irreducible affine. Let $\omega \in P^+$ be a twisting parameter,
and let $R\ge \log_q r$.
Let $Z(\bx)$ be a meromorphic function on $\Omega(R)$ satisfying
Hypothesis~\ref{hyp:analyticassumptions} and invariant under the
$W$-action $\vert_\omega^{\CG}$.

Then on the polydisk $P(R')$ with $R'=\max\{R,m_1^{-1},\ldots,m_r^{-1}\}$ we have the finite decomposition
\begin{align*}
Z(\bx)=\sum_{\mu\in \Pi_\om^{\max }}C_{\om-\mu}(\bx^\dd)\bx^{\omega - \mu}Z_\mu(\bx),
\end{align*}
where $C_{\om-\mu}(x)$ are power series in $x$ converging in a neighborhood of $0$.
\end{theorem}
\begin{proof}
 This is an immediate consequence of Theorem~\ref{the:basis} and the preceding discussion. The fact that
 $C_{\om-\mu}(x)$ converges in a neighborhood of $0$ follows from the fact that $C_{\om-\mu}(\bx^\dd)$ are subseries of $\Delta Z$,
 and from the fact that $Z$ satisfies the second assumption in Hypothesis~\ref{hyp:analyticassumptions}.
\end{proof}

As already mentioned in \S\ref{sec:untwisted}, for $\om=0$ we have $\Pi_0^{\max }=\{ 0 \}$, and this result was proved for
$\wD_4$ in \cite{BucurDiaconu2010Moments}, and in the affine simply-laced case
in \cite[Proposition 3.7]{Whitehead2014MultipleDS}.
\begin{remark}\label{rem:A11}
Another case when $\Pi_\om^{\max}=\{\om\}$ is when $\Phi$ is of type $\widetilde{A}_{r-1}$
and $\om=\om_i$ is a fundamental weight. Indeed, if $\mu\in \Pi_{\om_i}^{\max} $ then
$\la\mu,\dd\ra=\la\om_i,\dd\ra=1$ since $\dd$ is the sum of the simple roots in $\Phi$.
It follows that $\mu=\om_j+k\dd\le \om_i$ for some $j\in \{1,\ldots, r\}$ and $k\in \C$,
which is impossible unless $j=i$ and $k=0$.
\end{remark}

\section{A decomposition for twisted multiple Dirichlet series}\label{sec:WMDS}
We fix a symmetrizable Kac-Moody root system $\Phi$ of rank $r$.
We review the construction of the Weyl group multiple Dirichlet series $\mathcal{Z}(\bs;\bc)$ over  $\F_q(t)$, and its invariance properties under the Chinta-Gunnells action. We then show that the analytic conditions in
Theorem~\ref{the:basis} are satisfied, thus generalizing \cite[Theorem 4.4]{Friedlander2018TwistedWG} to the infinite dimensional Kac-Moody case.

\subsection{Construction of twisted MDS}
Fix $n\ge 2$ and a prime power $q \equiv 1 \pmod{2n}$, so that $\mathbb{F}_q$ contains the $n$-th roots of unity.
We start by defining two sets of Gauss sums, which will be assigned
to the extra formal variables $g_t$ in the Chinta-Gunnells action in the global and local settings.
Let $\left( \frac{a}{b}\right)$ be the $n$-th order residue symbol defined for $a, b \in \mathbb{F}_q[t]$--see \cite{Sawin2024GeneralMD} for its definition and main properties. The assumption on $q$ simplifies the reciprocity law to $\symb ab= \symb ba$ for $a,b$ coprime monic polynomials.

Let $\chi : \mathbb{F}_q^\times \to \mathbb{C}^\times$ be the $n$-th order character $\chi(a) = a^{\frac{q-1}{n}}$, fixing once and for all an identification $\mu_n(\F_q)\simeq \mu_n(\C)$.
Also let $\psi: \mathbb{F}_q \to \mathbb{C}$ be the additive character $\psi(x)= \exp(\frac{2\pi i}{p}\text{tr}_{\mathbb{F}_q/\mathbb{F}_p}x)$, where $p$ is the characteristic of $\mathbb{F}_q$. Then we define the associated Gauss sum
\begin{align*}
\gamma_t(q) = \sum_{a \in \mathbb{F}_q^\times}\chi(a)^t\psi(a)
\end{align*}
for $t \in \mathbb{Z}$. Note that $\gamma_t(q)$ depends only on the residue class of $t$ modulo $n$ and that $\gamma_0(q) = -1$ and $\gamma_t(q)\gamma_{-t}(q) = q$ for $t \not\equiv 0 \pmod{n}$.

For $\pi \in \mathbb{F}_q[t]$ monic irreducible, the associated Gauss sum is defined as
\be\label{eq:gammat}
\gamma_t(\pi) = \sum_{a \in (\mathbb{F}_q[t]/\pi)^\times}\left( \frac{a}{\pi}\right)^t\psi(\text{Res}(a/\pi)),
\ee
where $\text{Res}(f)$ for $f(t)$ a rational function is the coefficient of $t^{-1}$ when $f$ is expressed as a formal Laurent series. Note again that $\gamma_t(\pi)$ depends only on the residue class of $t$ modulo $n$ and that $\gamma_0(\pi) = -1$ and $\gamma_t(\pi)\gamma_{-t}(\pi) = |\pi|$ for $t \not\equiv 0 \pmod{n}$.

As in the introduction, we define a Weyl group multiple Dirichlet series with twisting parameter
$\bc=(c_1,\ldots, c_r)$  ($c_i\in \F_q[t]$ monic) for $\bs\in\hi$ by
\be\label{eq:globalwmds}
\Zc(\bs;\bc) = \sum_\boldm \frac{H(\boldm;\bc)}{|m_1|^{s_1}\cdots |m_r|^{s_r}}
\ee
where the sum is over tuples $\bm=(m_1,\ldots,m_r)$  of monic polynomials in $\F_q[t]$
and $|m|=q^{\deg m}$. The series depends only on $s_i=\a_i(\bs)$, so it can also be viewed
as a function on $\C^r$ as in the introduction. We will show that it converges for~$\Re(s_i)$ large enough.

The coefficients $H$ satisfy a twisted multiplicativity property \cite{Chinta2008ConstructingWG, Lee2012WeylGM}:
\begin{itemize}
 \item For $\bm=(m_1, \ldots, m_r)$, $\bm'=(m_1', \ldots, m_r')$  with $\prod m_i$ and $\prod m_i'$ coprime we have:
 \be\label{e_mult}
 \begin{aligned}
H(\bm \cdot \bm';\bc)&= H(\bm;\bc ) H(\bm';\bc)\cdot \\
 & \prod_{i=1}^{r} \symb{m_i}{m_i'}^{q(\alpha_i)} \symb{m_i'}{m_i}^{q(\alpha_i)}
\prod_{i<j} \symb{m_i}{m_j'}^{\la \a_i,\a_j \ra} \symb{m_i'}{m_j}^{\la \a_i,\a_j \ra}
\end{aligned}
\ee
where multiplication of tuples is defined componentwise.
\item For $\bc=(c_1, \ldots, c_r)$ and $\bc'=(c_1', \ldots, c_r')$ with $\prod m_i$ and $\prod c_i'$ coprime we have:
\[
H(\bm ;\bc\cdot \bc')=H(\bm ;\bc)\cdot \prod_{i=1}^r\symb{c_i'}{m_i}^{-q(\a_i)}.
\]
\end{itemize}
The series $\mathcal{Z}(\bs;\bc)$ is uniquely determined by its
 $\pi$-parts for  monic irreducible $\pi\in \F_q[t]$, which are defined as follows:
\be\label{eq:p_parts}
    \sum_{\lambda=\sum k_i\alpha_i \in Q^+}
    H(\pi^{k_1}, \ldots,\pi^{k_r} ;\pi^{v_\pi(c_1)}, \ldots, \pi^{v_\pi(c_r)})
    |\pi|^{-\lambda(\bs)} = Z_{\omega_\pi(\bc)}(|\pi|^{-\bs};|\pi|^{-1}).
 \ee
Here $\omega_\pi(\bc):= \sum_i v_\pi(c_i)\omega_i$, and
the parameters $g_t$ in the definition of $Z_{\omega_\pi(\bc)}(\bx;|\pi|^{-1})$ are given by
$g_0=-1$,
$g_t=\gamma_t(\pi)|\pi|^{-1}$ if $t\not\equiv 0 \pmod{n}$. We recall that for $\bs\in \hi$,
we set $s_i=\a_i(\bs)$ and $q^{-\bs}=(q^{-s_1}, \ldots, q^{-s_r})$. In particular,
for $\lam=\sum k_i \a_i\in Q$ we have $\lam(\bs)=\sum k_i s_i$.

\begin{remark}\label{rmk:qparameters}
For a local $\pi$-part, we use the substitution
$x_i=|\pi|^{-s_i}$, whereas the Chinta--Gunnells parameter is $|\pi|^{-1}<1$.
The convergence proofs in previous sections (e.g. in Lemma~\ref{lem:convergenceconstant})
use only that the exponential base is greater than $1$
and, separately, that the Chinta--Gunnells parameter is positive and different from $1$;
hence the same proofs apply to this $\pi$-part specialization.
\end{remark}

The choice of $\pi$-parts in defining $\mathcal{Z}(\bs;\bc)$ is far from being canonical. Even in the untwisted case (all $c_i=1$), when $\Phi$ is infinite there are infinitely many choices of $\pi$-parts invariant
under the untwisted Chinta-Gunnells action, as explained in Subsection~\ref{sec:untwisted}, and they would all produce an MDS satisfying the invariance in Theorem~\ref{thm:fullglobalconvergence} below.
The axiomatic method in~\cite{Diaconu2018ModuliOH, Whitehead2014MultipleDS, Sawin2024GeneralMD} produces a unique untwisted MDS
for arbitrary $\Phi$, and it is an open question how to choose the $\pi$-parts to recover the axiomatic MDS.

\subsection{Example: Kubota's Dirichlet series}\label{sec:Kubota}

The archetypal example of a twisted MDS is Kubota's Dirichlet series, which can be seen as the rank 1 example.
For monic $c,d$ and $\sym{\cdot}{\cdot}_m$ a residue symbol of order $m|n$, define the generalized Gauss sum
\[
 G(c,d)=\sum_{\substack{b\bmod d\\(b,d)=1}}
 \sym{b}{d}_m\psi\bigl(\Res(bc/d)\bigr),\qquad
 G(c,1)=1.
\]
The completed Kubota series is given by
\begin{equation}\label{eq:Kfull}
 \wD(s,c)=\frac{1}{1-q^{-m s}} \sum_{d\text{ monic}}\frac{G(c,d)}{|d|^{s+1}}.
\end{equation}
The coefficients $G(c,d)$ satisfy the same twisted multiplicativity relations
as the coefficients $H(d;c)$ above for the root system $A_1$, and the same is true
for the coefficients of $\wD(s,c)$. An easy computation shows that the $\pi$-part of
$\wD(s,c)$ is the Chinta-Gunnells average $Z_{v_\pi(c) \cdot \om_1}(|\pi|^{-s};|\pi|^{-1})$
for $A_1$, with $m=m_1$, $\sym{b}{d}_m=\sym{b}{d}^{q(\a_1)}$ and  the parameter
$g_{1+v_\pi(c)}=\gamma_{1+v_\pi(c)}(\pi)|\pi|^{-1}$ if
$m\nmid 1+v_\pi(c)$. We conclude therefore that
$$\Zc(s;c)= \wD(s,c).$$

The series $\wD(s,c)$ extends to a rational function in $x=q^{-s}$ on $\C$,
with $(1-qx^m) \wD(s,c)$ a polynomial in $x$.
Let $\wD^{(j)}(s,c)$ denote the function~\eqref{eq:Kfull}
with the sum restricted to $1+\deg c-2\deg d\equiv j\pmod m$--this is the $j$-signed
part of $\Zc(s;c)$.
By~\cite[Prop. 2.1]{Hoffstein1992ThetaFunctions} we have
\be\label{eq:KubFE}
\wD^{(j)}(s,c)=\frac{x^{\deg c +1}}{1-qx^m}
\bigl[ (1-q)x^{-j}\wD^{(j)}(-s,c) +\gamma_{-q(\a_1) j}(q) (1-x^{-m})\wD^{(-j)}(-s,c)\bigr],
\ee
namely $\wD(s,c)$ is invariant under the Chinta-Gunnells action with twist $\deg c$,
written using Lemma~\ref{lem:iparteq}.
In the next section we determine a region of convergence and prove the functional equation for the MDS
associated to a symmetrizable Kac-Moody root system.

\subsection{A domain of convergence and functional equations for twisted MDS}

For finite root systems, the functional equations of twisted MDS were
established over number fields by Chinta and Gunnells~\cite{Chinta2008ConstructingWG}
and over the rational function field by Friedlander~\cite{Friedlander2018TwistedWG}.
Lee and Zhang~\cite[Section 5]{Lee2012WeylGM} extended the number field
case to symmetrizable Kac--Moody root systems. Their argument derives
a region of absolute convergence from the functional equations using
Bochner's analytic continuation principle and Phragm\'en--Lindel\"of
bounds for Kubota's Dirichlet series. In the function field case, we take advantage of the fact that Kubota's series
are rational functions to give a simpler proof of convergence in Theorem~\ref{thm:fullglobalconvergence}.

\begin{lemma}
\label{lem:deltaeulerproduct}
Let $\bs \in \hi$ with $\Re(s_i)>\max\{ \log_q r, m_1^{-1},\ldots, m_r^{-1}\}$. We have
\begin{align*}
    D(q^{-\bs};q) = \prod_\pi\Delta(|\pi|^{-\bs})
\end{align*}
where the product runs over all monic irreducible $\pi \in \mathbb{F}_q[t]$.
\end{lemma}
\begin{proof}
The infinite product defining $D(q^{-\bs};q)$ converges absolutely for $\Re(s_i)>\log_q r$ by Lemma~\ref{lem:setY} (a).
If in addition $\Re(s_i)>m_j^{-1}$ for all $j$,  we have $\Re(m_\alpha\alpha(\bs))>1$ for all $\a \in \Phi^+$, and we obtain
\begin{align*}
\prod_\pi \Delta(|\pi|^{-\bs}) &= \prod_{\alpha \in \Phi^+}\prod_\pi (1-|\pi|^{-m_\alpha\alpha(\bs)})^{\mult(\a)}
\\
&= \prod_{\alpha \in \Phi^+}(1-q^{1-m_\alpha\alpha(\bs)})^{\mult(\a)}
\end{align*}
by the formula for the zeta function for $\mathbb{F}_q(t)$.
\end{proof}
For a power series $f(\bx)=\sum a_\lam \bx^\lam$, we denote by
$f^\abs(\bx)=\sum |a_\lam\bx^\lam|$ its absolute majorant.
We recall also the open subset $C(R)$ of the fundamental Weyl chamber and the associated
convex, $W$-invariant subset $X(R)$ of the Tits cone, defined for $R\ge 0$ in Subsection~\ref{sec:convdomain}.

\begin{proposition}
\label{thm:globalfundemntalchamber}
The series $D(q^{-\bs};q) \mathcal{Z}(\bs;\bc)$
converges absolutely for $\bs \in C(R_{q,r})_\C$ where
$$R_{q,r}=\max (1,\log_q r) .$$
\end{proposition}
\begin{proof}
Let $\Zm(\bx;\bc)$ be the series $\mathcal{Z}(\bs;\bc)$ after changing variables $q^{-s_i}\mapsto x_i$,
$i=1,\ldots, r$, so that $\mathcal{Z}(\bs;\bc)=\Zm(q^{-\bs};\bc)$.
Using Lemma~\ref{lem:deltaeulerproduct}, it can be shown
that the function $D(q^{-\bs};q) \Zc(\bs;\bc)$
is a MDS with $\pi$-part $Z^*_{\om}(|\pi|^{-\bs};|\pi|^{-1})$ with $\om=\omega_{\pi}(\bc)$.
Since its corresponding $H$-coefficients become multiplicative after replacing each term in the MDS
by its absolute value,  we obtain, taking $\bs\in \hi_\R$ without loss of generality
\begin{align*}
\bigl[D(\bx;q)\Zm(\bx;\bc)\bigr]^\abs(q^{-\bs};q) \le \prod_\pi \big(Z^{*}_{\om}\big)^\abs(|\pi|^{-\bs};|\pi|^{-1}).
\end{align*}
Assume $s_i>1$ and set $S = \min_i\{s_i\}>1$.
We have
\begin{align*}
\big(Z^*_{\om}\big)^\abs (|\pi|^{-\bs};|\pi|^{-1}) &\leq\sum_{\sigma \in W}
(1\vert_\omega \sigma)^\abs( |\pi|^{-\bs};|\pi|^{-1} ) \\
&\leq 1 + \sum_{k=1}^\infty \left(\frac{3r |\pi|^{-S}}{1-|\pi|^{-1} }\right)^k,
\end{align*}
where we used Lemma~\ref{lem:supp} with $M=|1-|\pi|^{-1}|+2 \max_t |g_t|<3$
as $|g_t|=1/\sqrt{|\pi|}$ if $t\not\equiv 0 \pmod{n}$,
together with the crude bound
$|\pi|^{-(\theta-\sigma^{-1}\theta)(\bs)}\leq |\pi|^{-S \ell(\sigma)}$.
In particular, for $|\pi|$ large enough (for example $S > \log_{|\pi|}(6r)$) this is a convergent geometric series, so that we have
\begin{align*}
\sum_{\sigma \in W} (1\vert_\omega \sigma)^\abs = 1 + O(|\pi|^{-S}).
\end{align*}

For the finitely many primes $\pi$ such that  $S \leq \log_{|\pi|}(6r)$
the corresponding $\pi$-part $Z^*_{\om}(|\pi|^{-\bs};|\pi|^{-1})$ converges absolutely by Theorem~\ref{thm:convergence} (with the obvious changes described in Remark~\ref{rmk:qparameters},
and noting that $D_\re(|\pi|^{-\bs};|\pi|^{-1})\ne 0$ for $s_i>1$). In particular all the factors in the Euler product above
are absolutely convergent, and all but finitely many are bounded uniformly by $1 + O(|\pi|^{-S})$. Since $S>1$, the series $\sum_\pi |\pi|^{-S}$ converges, finishing the proof.
\end{proof}

\begin{theorem}
\label{thm:fullglobalconvergence}
Assume $q\equiv 1\pmod{2n}$.
The function $D_\re(q^{-\bs};q) \mathcal{Z}(\bs;\bc)$ converges normally on $X(R_{q,r})_\C$.
Moreover, in this region $\mathcal{Z}(\bs;\bc)$ is invariant under the Chinta-Gunnells
action $\vert_\om^\CG$ of the Weyl group $W$, with twisting parameter $\om=\sum \deg(c_i) \om_i$,
$x_i=q^{-s_i}$ and $g_t=\gamma_t(q)$ for $t\not\equiv 0 \pmod n$.
\end{theorem}
\begin{proof}
For finite $\Phi$, the Chinta-Gunnells invariance was proved in~\cite{Chinta2008ConstructingWG, Friedlander2018TwistedWG}
by studying a rank one
subseries of the MDS, that can be expressed in terms of Kubota's Dirichlet series.
The same argument applies in our setting, so we only give a sketch below,
adapted to our normalization and to our choice of $\pi$-parts.

For $\bd=(d_1,\ldots,d_r)$ a tuple of monic polynomials, denote by
$\wbd$ 
the tuple with its $i$th entry removed. Set $m=m_i$. For $0\le j<m$, the $j$-signed part of $\Zc$ from Definition~\ref{def:j_signed}
is given by
\be\label{eq:Zij_decomp}
\Zc_i^{(j)}(\bs;\bc)=
\sum_{\wbd=(d_1,\ldots,d_{i-1},d_{i+1},\ldots,d_r)} \prod_{k \ne i}|d_k|^{-s_k} \Ec^{(j)}(s_i;\wbd),
\ee
where $\Ec^{(j)}(s_i;\wbd) =\Ec^{(j)}(s_i;\wbd;\bc)$ is given by
\begin{equation}\label{eq:EB}
 \Ec^{(j)}(s;\wbd)=
 \sum_{\substack{d_i \text{ monic}\\
                1+\deg B-2\deg d_i \equiv j\pmod m}}
 H(\bd;\bc)|d_i|^{-s}, \qquad B=B_{\wbd}=c_i\prod_{k\sim i}d_k^{-a_{ik}}.
\end{equation}
Since $\wbd$ is fixed from now on, we omit it from the notation of the various quantities involved.
Let $S$ be the set of primes dividing $\prod_{j\ne i}d_j\prod_j c_j$.

As in \cite[eq. (33)]{Friedlander2018TwistedWG}, we have a decomposition in terms of the Kubota
series with $m$-th order character $\sym{\cdot}{\cdot}_{m}:=\sym{\cdot}{\cdot}^{q(\a_i)}$ introduced in \S\ref{sec:Kubota}:
\be \label{eq:full}
\Ec^{(j)}(s;\wbd)=\sum_{b}C(b)G_b(s) \wD^{(j)}(s,B_b), \qquad B_b:=\prod_{\pi\in S}\pi^{r_m(v_\pi(B)-2v_\pi(b))},
\ee
where $b$ ranges over the monic $m$-th power-free polynomials supported on the
set $S$. Here $G_b(s)$ is the Dirichlet polynomial
\[
G_b(s)=\prod_{\pi\in S}
 \bigl[\Pc_{\pi}^{(j_\pi(v_\pi(b)))}\bigr]_{[v_\pi(b)]}
                   (|\pi|^{-s}),\qquad j_\pi(k)=r_{m}(1+v_\pi(B)-2k),
\]
where $\Pc_\pi^{(j)}$ are the polynomials $\Pc_{\lam_\pi(\wbd)}^{(j)}$ with $\lam_\pi(\wbd)=\sum_{k\ne i} v_\pi(d_k)\a_k$,
associated as in \S\ref{sec:local_poly} with the $\pi$-part $Z_{\om_\pi(\bc)}(\bx;|\pi|^{-1})$ with
parameters $g_t=\gamma_{t}(\pi)/|\pi|$ for $t\not\equiv 0\pmod n$.
For a polynomial $P(x)=\sum a_n x^n$ we denote by $[P]_{[k]}(x)=\sum_{n\equiv k \pmod m} a_n x^n$.
Finally, the function
\begin{equation}\label{eq:C}
 C(d_i)=\prod_{\pi\in S}\prod_{j=1}^r
 \sym{c_j/\pi^{v_\pi(c_j)}}{\pi^{v_\pi(d_j)}}^{-q(\alpha_j)}
 \cdot\prod_{\substack{\pi,\rho\in S\\\pi\ne \rho}}
 \sym{\pi}{\rho}^{\langle\lambda_\pi(\bd),\lambda_\rho(\bd)\rangle/2}
\end{equation}
 satisfies $C(d_i)=C(b)$ if $d_i=ba^m$ with $a,b\in\F_q[t]$ monic.

The functional equation
$\Zc_i^{(j)}=(\Zc\vert_\om^\CG \ss_i)_i^{(j)}$ (with the right-hand side expressed as in Lemma~\ref{lem:iparteq}) follows from~\eqref{eq:full},
by Kubota's
functional equation~\eqref{eq:KubFE} combined with the local functional equations~\eqref{eq:local-polynomial-FE}
for $\Pc_\pi^{(j)}$.
Since $\wD(s,c)$ is a rational function, this identity is understood as an identity of
power series in $x_k$, $k\ne i$, with rational function coefficients in $x_i$. The
meromorphic identity on $X(R_{q,r})_\C$ follows from the normal convergence proved below.

Let $F(\bs)=D_\re(q^{-\bs};q) \Zc(\bs;\bc)$, so $F$ converges absolutely on $C(R_{q,r})_\C$ by
Proposition~\ref{thm:globalfundemntalchamber} and Lemma~\ref{lem:setY} (b). To prove normal convergence in $X(R_{q,r})_\C$,
by Lemma~\ref{lem:conv} it suffices to show that the convergence domain is $W$-invariant.
Using the functional equation just proved and~\eqref{eq:Dre_sigmai}, we have
\be\label{eq:F_funeq}
 (1-qx_i^{m})F(\sigma_i \bs)
 =x_i^{-l_i-1}\sum_{j=0}^{m-1}F_i^{(j)}(\bs)
 \left[(1-q)x_i^j-(x_i^{m}-1)\gamma_{q(\alpha_i)j}(q)\right],
 \qquad x_i=q^{-s_i}.
\ee
The Kubota series is a rational function with
$(1-q x_i^{m})\wD^{(j)}(s_i,a)$ a polynomial in $x_i$,
so the relations~\eqref{eq:Zij_decomp} and~\eqref{eq:full} show that
\[
F(\bs)=\sum_{\wlam=\sum_{j\ne i}k_j\a_j \in Q^+} P_{i,\wlam}(x_i) q^{-\wlam (\bs)}
\]
with polynomials $P_{i,\wlam}(x_i)$. Writing $F(\bs)=\sum b_\lam q^{-\lam(\bs)}$, for $\bs\in \hi_\R$ define
$F^\abs(\bs)=\sum |b_\lam| q^{-\lam(\bs)}$. For a Laurent polynomial $P(t)$ and $t>0$ we have
\[
\bigl[(1-qt^m)P(t)\bigr]^\abs\ge |1-qt^m| P^\abs(t)
\]
and we conclude from~\eqref{eq:F_funeq} that for real $\bs\in \hi_\R$ we have
\[
 |1-qx_i^{m_i}|F^{\mathrm{abs}}(\ss_i\bs) \le M_i(x_i)F^{\mathrm{abs}}(\bs)
\]
where $M_i(x)=x^{-l_i-1}
 \max_{0\le j<m_i}\left((q-1)x^j+(1+x^{m_i})|\gamma_{q(\alpha_i)j}(q)|\right)$.

 Let $U$ be the interior of the domain of absolute convergence for $F(\bs)$.
 Every $\bs\in U\cap\hi_\R$ is the midpoint of two points
$\bs',\bs''\in U\cap\hi_\R$ with $s_i',s_i''\ne 1/m_i$.
Applying the preceding inequality on sufficiently small neighborhoods of these points gives $\ss_i\bs',\ss_i\bs''\in U$.
By convexity (Lemma~\ref{lem:conv}), their midpoint $\ss_i\bs$ also lies in $U$.
Thus $U$ is invariant under every simple reflection, completing the proof.
\end{proof}

\begin{corollary}
\label{thm:globalbasis}
Let $\bc \in \mathbb{F}_q[t]^r$ be a twisting parameter, and set $\omega = \sum_{i=1}^r \deg(c_i)\omega_i$.
Then for $\bs \in C(R_{q,r})_\C$ with $R_{q,r}=\max(1,\log_q r)$, we have
\begin{align*}
\mathcal{Z}(\bs;\bc) = \sum_{\xi \in\Pi_\omega}c_{\om-\xi} q^{-(\omega - \xi)(\bs)}Z_\xi(q^{-\bs};q),
\end{align*}
where $c_\lam$ is the coefficient of $q^{-\lam(\bs)}$ in the expansion of $\Delta(q^{-\bs}) \mathcal{Z}(\bs;\bc)$.
\end{corollary}
\begin{proof}
By the previous theorem, the function $D_\re(q^{-\bs};q) \mathcal{Z}(\bs;\bc)$ is holomorphic on $X(R_{q,r})_\C$.
The second assumption in Hypothesis~\ref{hyp:analyticassumptions} is satisfied by the function $\mathcal{Z}(\bs;\bc)$ after the change of variables
$q^{-\bs}\mapsto \bx$, by Proposition~\ref{thm:globalfundemntalchamber} and the fact
that $D_\re(q^{-\bs};q)$ is nonzero on $C(R_{q,r})_\C$. Therefore we can apply Theorem~\ref{the:basis} for $R=R_{q,r}$ to finish the proof.
\end{proof}

\section{Multiple Dirichlet series for \texorpdfstring{$\affa{1}$}{A1}}
\label{section:computinga1avg}
The goal of this section is to illustrate the use of the decomposition in Theorem~\ref{the:decompaaffine} in the simplest case, where $\Phi$ is of type $\affa{1}$
and $n=2$. We compute the Chinta-Gunnells averages for $\affa{1}$ with twisting parameters $\omega=0$ and $\omega =\omega_i$,
where $\om_i$ is a fundamental weight. In both cases, there are natural candidate functions $Z$ invariant under the $\vert_\om^\CG$-action, and
$\Pi_\om^{\max}=\{\om\}$, so  the difficult part of applying Theorem~\ref{the:decompaaffine}
is to determine the power series coefficient $C_0(x)$ appearing there. For this we show the existence of an extra functional equation
satisfied by $Z_\om$ that does not come from the Weyl group. This type of extra functional equation was first observed in the untwisted case for $\Phi$ of type $\widetilde{D}_4$ in \cite{diaconu2023quadraticweylgroupmultiple}.

The explicit computation of $Z_0$  yields a $q$-deformation of the classical Jacobi triple product formula,
described in the introduction. For $Z_{\om_i}$, we obtain a $q$-deformation of the character of the module with highest weight $\om_i$. The method of proving these identities--by means of a functional equation not arising from the underlying Weyl group--may be of independent interest in the field of combinatorial hypergeometric series identities.

In the final subsection, we consider the twisted WMDS $\Zc(\bs;\bc)$  over $\F_q[t]$,
for square-free twisting polynomials. In the case of $\affa{1}$, this MDS is an Euler
product that we compute explicitly in Theorem~\ref{thm:explicitaffa1mds}, using the formulas for $Z_0$ and $Z_{\om_i}$.
Interestingly, we show that $\Zc(\bs;\bc)$ also satisfies a global version of the extra functional equation (Theorem~\ref{thm:tauactiononmds}).

\subsection{The Chinta-Gunnells averages \texorpdfstring{$Z_0$}{Z0} and \texorpdfstring{$Z_{\om_1}$}{Z1}}
Recall that $\affa{1}$ is the root system with (generalized) Cartan matrix  $\bigl( \begin{smallmatrix}2 & -2\\ -2 & 2\end{smallmatrix}\bigr)$. We denote the simple roots
by $\a_0, \a_1$. The smallest positive imaginary root is $\delta = \alpha_0 + \alpha_1$ and we have
$\Phi_{\re}^+ = \{\alpha_i + n \delta\mid i=0,1,\  n \geq 0\}$.

In this section, we fix $n=2$, $g_1=u$ and $q(\a_i)=1$ in the definition of the Chinta-Gunnells action, so that $q=u^2$.
For a twisting parameter $\omega = l_0\omega_0 + l_1 \omega_1\in P^+$,
the Chinta-Gunnells action of generators can be expressed explicitly on functions as
\begin{align*}
f\vert^{\CG}_\omega \sigma_i (\bx)= x_i^{l_i} \cdot
\begin{cases}
\frac{1-u/x_i}{1-ux_i}f(\sigma_i\bx), & l_i \text{ even}\\
\frac{1}{x_i}f(\sigma_i\bx), & l_i \text{ odd}.\\
\end{cases}
\end{align*}
We also have the related action $f|_\om \ss_i (\bx)=-x_i^2 f\vert_\om^\CG\ss_i (\bx)$.

We define the Chinta-Gunnells average of a Laurent polynomial $f(\bx;u)$ by
\begin{equation}\label{eq:genericaffinea1average}
Z_{\omega ,f}(\bx;u) =  \frac{\sum_{\sigma \in W}f \vert_{\omega} \sigma (\bx;u)}{\Delta(\bx)},
\end{equation}
so that $Z_{\omega,1} = Z_\omega$. For a simply-laced affine $\Phi$ and $\om=0$, it was shown in~\cite[Prop. 3.2]{diaconu2023quadraticweylgroupmultiple}
that there is a sufficiently large integer $M=M_f$ such that the function $\bx^{M\dd}Z_{0 ,f}(\bx;u)$ converges absolutely in the interior of the Tits cone, except for simple poles at $u \bx^\a=\pm 1$ for $\a \in \Phi^+_{\re}$. The same conclusion holds for the twisted action by a similar proof.

In the next two subsections, we compute $Z_\om$ explicitly
for $\om\in \{0, \om_1\}$. In both cases, we have
$\Pi_\om^{\max}=\{\om\}$ in Theorem~\ref{the:decompaaffine} (see Remark~\ref{rem:A11}),
and we apply the theorem to the function $Z=F_\om$ where
\[
F_\om(\bx;u)=
\begin{cases}
\prod_{\alpha \in \Phi_{\re}^+}(1-u\bx^{\alpha})^{-1} & \text{ if } \om=0\\
\prod_{\alpha \in \Phi_{\re}^+ \cap W \alpha_0}(1-u\bx^{\alpha})^{-1}   & \text{ if } \om=\om_1.
\end{cases}
\]
More explicitly, we have
\[
F_0(\bx;u)=\frac{1}{\qpoc{ux_0}{\bx^{\dd}}\qpoc{ux_1}{\bx^{\dd}} }, \qquad
F_{\om_1}(\bx;u)=\frac{1}{\qpoc{ux_0}{\bx^{2\dd}}\qpoc{ux_1\bx^\dd}{\bx^{2\dd}} }
\]
and one checks directly that $F_\om=F_\om |_\om^\CG \ss_i$, $i=0,1$. Clearly
$F_\om$ converges in the interior of the Tits cone $|\bx^\dd|<1$, so we can apply Theorem~\ref{the:decompaaffine} to conclude that
$F_\om(\bx;u)= D_\om(\bx^\dd;u) Z_\om(\bx;u) $
for a power series $D_\om$ in $x$ whose coefficients are polynomials in $u$.
Let $N_\om=D_\om^{-1}$, so that
\begin{equation}\label{eq:untwistedA1average}
Z_\om(\bx;u)= N_\om(\bx^\dd;u) F_\om(\bx;u) .
\end{equation}
Note that $N_\om(x;u)$ converges normally for $|x|<1$, as the poles of $Z_\om$
are easily seen to be among those of $F_\om$, by the explicit formula for the action above. Also its coefficients
are polynomials in $u$.  Although $Z_\om(\bx;u)$ is defined for $q=u^2>1$, its coefficients
are polynomials in $u$ and only finitely many $w\in W$ in the series contribute to any given coefficient.
Therefore one can specialize $u=-1$ or $u=0$ below.

It remains to explicitly determine the power series $N_\om(x;u)$.
The appearance of such a factor is a recurring theme in the theory of affine Kac--Moody algebras, beginning with Macdonald's classical formula \cite{MacDonald1971AffineRS}. Determining this factor is subtle in general, and Macdonald accomplished it case by case, by means of suitable specializations of the variables.

In the next two subsections, we determine the factor $N_\om$ using two ingredients:
\begin{itemize}
\item an extra functional equation satisfied by $Z_\om(\bx;u)$, under the transformation $u\mapsto u \bx^\dd$ for $\om=0$ and $u\mapsto u\bx^{2\delta}$ for $\om=\om_1$;
\item a specialization at $u=-1$ or $u=0$, depending on whether $\om=0$ or $\om=\om_1$, respectively,
 where $Z_\om(\bx;u)$ can be explicitly determined from Macdonald's formula:
\begin{align}
\label{eq:macdonald}
\sum_{\sigma \in W}(-1)^{\ell(\sigma)}\bx^{\rho-\sigma^{-1}\rho }=
\prod_{\alpha \in \Phi_{\re}^+}(1-\bx^{\alpha})\prod_{n\geq1}(1-\bx^{n\delta}).
\end{align}
\end{itemize}

The existence of an extra functional equation in the untwisted affine setting was discovered in \cite{diaconu2023quadraticweylgroupmultiple} for $\wD_4$
with the help of a computer.
In the case of $\wA_1$, we establish the extra functional equation for $Z_0$ and $Z_{\om_1}$ in the next two subsections.

\subsection{Extra functional equation satisfied by \texorpdfstring{$Z_{0}$}{Z0}}
\label{subsec:untwistedaffa1computation}

We first consider the untwisted Chinta-Gunnells average $Z_{0}(\bx;u)$,
which we denote in this subsection by $Z(\bx;u)$ for simplicity.

If $f$ is a power series in $\bx$ and $u$ we define
\begin{align*}
f^\tau(\bx;u) = f(\bx;u\bx^\delta).
\end{align*}

We show that there exists a rational function $B(\bx;u)$ such that the action
\begin{align*}
(f\vert_0 \tau)(\bx;u) = B(\bx;u)f^\tau(\bx;u)
\end{align*}
commutes with the Chinta-Gunnells action $\vert_0$ of the Weyl group, namely
$f\vert_0 \tau\vert_0 \sigma_i = f\vert _0\sigma_i\vert_0 \tau$, $i=0,1$. For simplicity,
we omit the dependence on the weight $\om=0$ and write
$f\vert \tau$ for $f\vert_0 \tau$.

A computation shows that this requirement is equivalent to the equality
\begin{align*}
\frac{B(\bx;u)}{B(\sigma_i\bx;u)}=\frac{(1-u/x_i)(1-ux_i\bx^{\delta})}{(1-ux_i)(1-u\bx^{\delta}/x_i)},
\end{align*}
which has the solution
\begin{align*}
B(\bx;u) = \frac{1}{(1-ux_0)(1-ux_1)}.
\end{align*}

Note that the inverse action $\tau^{-1}$
is given by
\begin{align*}
(f\vert \tau^{-1})(\bx;u) = A(\bx;u)f(\bx;\frac{u}{\bx^{\delta}})
\end{align*}
where
\begin{align*}
A(\bx;u) = B^{-1}\left(\bx;\frac{u}{\bx^{\delta}}\right)=\frac{(x_0-u)(x_1-u)}{\bx^\delta}.
\end{align*}
It is important to note that $1\vert\tau^{-1}(\bx;u)=A(\bx;u)$ is a Laurent polynomial in $\bx$,
as in the case of $\wD_4$ \cite[Sec. 4.1]{diaconu2023quadraticweylgroupmultiple}.

Since $\Delta$ does not depend on $u$, the commutativity of $\tau$ and the Weyl group action gives
$$Z_f\vert \tau = Z_{f\vert \tau}$$
for an arbitrary Laurent polynomial $f$ such that $f\vert \tau$ is also a Laurent polynomial, where $Z_f=Z_{0,f}$ was defined in~\eqref{eq:genericaffinea1average}.
In particular, this is the case with $f = 1\vert\tau^{-1}$, and obtain the functional equation
\begin{align}
\label{eq:untwistedfunceq}
Z = Z_{1\vert \tau^{-1}}\vert\tau.
\end{align}
To compute $Z_{1\vert \tau^{-1}}$, we note that
$Z_u = uZ$, $Z_{\bx^{\delta}} = \bx^{\delta}Z$, and
 $(x_0 - u)\vert\sigma_0 = u-x_0$, which gives
 $Z_{x_0-u} = 0$. Using the formula above for $1\vert \tau^{-1}=A$, we obtain
\begin{align*}
Z_{1\vert \tau^{-1}} =Z \cdot (1-u^2/\bx^\dd).
\end{align*}
The functional equation in (\ref{eq:untwistedfunceq}) now becomes
\begin{align*}
Z_0(\bx;u) = Z_0(\bx;u\bx^{\delta})\cdot \frac{1-u^2\bx^{\delta}}{(1-ux_0)(1-ux_1)}.
\end{align*}
Writing $Z=Z_0$ in the form of (\ref{eq:untwistedA1average}) and cancelling the denominators,
we obtain the functional equation for $N=N_0$:
\begin{align}
\label{eq:nfunceq}
N(x;u)= (1-u^2x)N(x;ux).
\end{align}

Finally, we can use the Macdonald identity to compute explicitly a specialization of $N(x;u)$.
Setting $u=-1$ in (\ref{eq:genericaffinea1average}), and applying Macdonald's identity~(\ref{eq:macdonald}) we have
\begin{align*}
Z(\bx;-1) &= \frac{\sum_{\ss\in W}(-1)^{\ell(\ss)}\bx^{\rho-\ss^{-1}\rho}}{\Delta}\\
&= \frac{\prod_{\alpha \in \Phi_{\re}^+}(1-\bx^{\alpha})\cdot\prod_{n\geq1}(1-\bx^{n\delta})}{\prod_{\alpha \in \Phi_{\re}^+}(1-\bx^{2\alpha})\cdot \prod_{n \geq 1} (1-\bx^{2n\delta})}
\end{align*}
so that, from (\ref{eq:untwistedA1average})
\[
N(x;-1) = \frac{1}{\prod_{n\geq 1}(1+x^n)}= (x;x^{2})_\infty.
\]

The ansatz
\begin{align*}
N(x;u) = (u^2x;x^{2})_{\infty}
\end{align*}
satisfies both the functional equation and the specialization at $u=-1$,
leading to the main result of this section.

\begin{theorem}
\label{the:explicituntwistedaverage}
We have
\begin{align*}
Z_0(\bx; u) = \frac{(u^2\bx^{\delta};\bx^{2\delta})_{\infty}}
{(ux_0;\bx^{\delta})_\infty(ux_1;\bx^{\delta})_\infty}.
\end{align*}
\end{theorem}
\begin{proof}
Clearly this is equivalent to showing that $N(x;u) = (u^2x;x^{2})_{\infty}$.
For $|xu^2|<1$, $|x| <1$ we have a power series expansion
\begin{align*}
\frac{N(x;u)}{(u^2x;x^2)_{\infty}} = \sum_{n,m\geq 0}c_{n,m}u^mx^{n}
\end{align*}
for $c_{n,m} \in \mathbb{C}$.
Note that we may apply the functional equation $u\mapsto ux$ to both sides since after applying it we remain inside the region of convergence.
 By~\eqref{eq:nfunceq}, the left-hand side is invariant under this transformation,
so we conclude that $c_{n,m} = c_{n-m,m}$.  Iterating this relation we obtain $c_{n,m} = 0$ for $m \neq 0$, as $c_{n,m} = 0$ for $n<0$. Therefore
the right-hand side is constant in $u$, so it equals 1 by the specialization above.
\end{proof}

By writing the Chinta-Gunnells action explicitly, we can restate the previous theorem
as a $u$-deformation of the Jacobi triple product identity (the case $u=-1$).
\begin{corollary}
\label{cor:untwisteddeformationformula}
Let $x_0, x_1 \in\mathbb{C}$ and set $\bx^{\delta} = x_0x_1$. Then
\begin{align*}
\sum_{n\in\mathbb{Z}} u^nx_0^{\frac{n(n+1)}{2}}x_1^{\frac{n(n-1)}{2}}\frac{(u^{-1}x_0;\bx^{\delta})_n}{(ux_0;\bx^{\delta})_n} = \frac{(u^2\bx^{\delta};\bx^{2\delta})_\infty(x_0^2;\bx^{2\delta})_\infty(x_1^2;\bx^{2\delta})_\infty(\bx^{2\delta};\bx^{2\delta})_\infty}{(ux_0;\bx^{\delta})_\infty(ux_1;\bx^{\delta})_\infty}.
\end{align*}
\end{corollary}
\begin{proof}
An induction on $n=\ell(\sigma)$ shows that if $\sigma \alpha_i < 0$, then
\begin{align*}
1 \vert_0 \sigma =u^nx_i^{n}\bx^{\frac{n(n-1)}{2}\delta}\frac{(u^{-1}x_i;\bx^{\delta})_n}{(ux_i;\bx^{\delta})_n}.
\end{align*}
Recalling that for $n>0$
\begin{align*}
(a;b)_{-n}={\frac {(-b/a)^{n}b^{n(n-1)/2}}{(b/a;b)_{n}}}
\end{align*}
one can rewrite $\sum_{\ss\in W} 1|_0 \ss$ as in the left-hand side of the identity.
\end{proof}
\subsection{Extra functional equation satisfied by \texorpdfstring{$Z_{\omega_1}$}{Zomega1}}
\label{subsec:twistedaffA1}
We now consider the twisted Chinta-Gunnells average $Z_{\omega_1}$ (again denoted in this section by $Z$).

In this case there is no obvious action $f|_{\om_1} \tau$ as in the previous subsection,
and we look for an action associated to the transformation
$f^{\tau^2}(\bx;u)=f(\bx;u\bx^{2\dd})$:
\begin{align*}
(f\vert_{\omega_1} \tau^2)(\bx;u) = B_{\om_1}(\bx;u)f^{\tau^2}(\bx;u),
\end{align*}
commuting with the Chinta-Gunnells action:
$$
f\vert_{\omega_1} \tau^2\vert_{\omega_1} \sigma_i = f\vert_{\omega_1} \sigma_i\vert_{\omega_1} \tau^2 \qquad (i=0,1).
$$
The commutativity is equivalent to the relations
\[
\frac{B_{\om_1}(\bx;u)}{B_{\om_1}(\sigma_0\bx;u)}=\frac{(1-u/x_0)(1-ux_0\bx^{2\delta})}{(1-ux_0)(1-ux_1\bx^{\delta})},\qquad
B_{\om_1}(\bx;u)=B_{\om_1}(\sigma_1\bx;u),
\]
which are satisfied by
\begin{align*}
B_{\om_1}(\bx;u) = \frac{1}{(1-ux_0)(1-ux_1\bx^{\delta})}.
\end{align*}
The inverse action is given by
$(f\vert_{\omega_1} \tau^{-2})(\bx;u) = A_{\om_1}(\bx;u)f(\bx;\frac{u}{\bx^{2\delta}}),$
with
\[
A_{\om_1}(\bx;u) = B_{\om_1}^{-1}\left(\bx;\frac{u}{\bx^{2\delta}}\right)=\frac{(x_0-u)(x_1\bx^{\delta}-u)}{\bx^{2\delta}},
\]
which is a Laurent polynomial.

As before, we obtain a functional equation (recall $Z=Z_{\om_1}$ here)
\begin{align}
\label{eq:twistedfunceq}
Z = Z_{1\vert_{\omega_1} \tau^{-2}}\vert_{\omega_1}\tau^2.
\end{align}

To compute $Z_{1\vert_{\omega_1} \tau^{-2}}$, we note that
$(x_0 - u)\vert_{\omega_1}\sigma_0 = u-x_0$, so  $Z_{x_0-u} = 0$ and
$Z_{x_0} = Z_{u} = uZ$.  Also  $x_1 \vert_{\omega_1} \sigma_1 = -x_1$,
so $Z_{x_1} = 0$. It follows that
$Z_{1\vert_{\omega_1} \tau^{-2}}=Z $, and the functional equation in (\ref{eq:twistedfunceq}) now becomes
\begin{align*}
Z_{\om_1}(\bx;u) = Z_{\om_1}(\bx;u\bx^{2\delta})\cdot \frac{1}{(1-ux_0)(1-ux_1\bx^{\delta})}.
\end{align*}

Writing $Z$ as in (\ref{eq:untwistedA1average}) and cancelling denominators,
we obtain a functional equation for $N:=N_{\om_1}$:
\begin{align*}
    N(x;u) = N(x;ux^{2}).
\end{align*}

Again, we compute a specialization of $N(x;u)$, this time at $u=0$.
In this case we have
$
\bx^{\lambda}\vert_{\omega_1}\sigma_i = -x_i^2\bx^{\sigma_i\lambda}
$,
and therefore, applying Macdonald's formula (\ref{eq:macdonald}), we obtain
\begin{align*}
Z_{\om_1}(\bx;0) &= \frac{\sum_{\sigma \in W}(-1)^{\ell(\sigma)}\bx^{2(\rho-\sigma^{-1}\rho)}}{\Delta}=1.
\end{align*}
From (\ref{eq:untwistedA1average}) we have $N(x;0) = 1$.
An argument similar to the proof of~Theorem~\ref{the:explicituntwistedaverage}
then gives $N(x;u) =1$, proving the following:
\begin{theorem}
\label{the:explicittwistedaverage}
We have
\begin{align*}
Z_{\omega_1}(\bx; u) = \frac{1}{\qpoc{ux_0}{\bx^{2\dd}}\qpoc{ux_1\bx^\dd}{\bx^{2\dd}} }.
\end{align*}
\end{theorem}

Writing the formula for the twisted action explicitly, we obtain a
$u$-deformation of the character of the representation with highest weight $\om_1$
(the case $u=-1$).

\begin{corollary}
\label{cor:twistedpartdeformationformula}
Let $x_0, x_1 \in\mathbb{C}$ and set $\bx^{\delta} = x_0x_1$. Then
\begin{align*}
\sum_{n\in\mathbb{Z}} (-u)^nx_0^{3n^2+2n}x_1^{3n^2-n}(1-x_0^{-4n}x_1^{2-4n})\frac{(u^{-1}x_0;\bx^{2\dd})_n}{(ux_0;\bx^{2\dd})_n}=\frac{(x_0^2;\bx^{2\delta})_\infty(x_1^2;\bx^{2\delta})_\infty(\bx^{2\delta};\bx^{2\delta})_\infty}{(ux_0;\bx^{2\delta})_\infty(ux_1\bx^{\delta};\bx^{2\delta})_\infty}.
\end{align*}
\end{corollary}
\begin{proof}
To show that $\sum_{\ss\in W} 1|_{\om_1} \ss $ is the left-hand side of the identity above,
we consider four cases.

Assume $\sigma \alpha_0 < 0$. An induction on length shows that if $\ell(\sigma) = 2n$, then
\begin{align*}
1 \vert_{\omega_1} \sigma = (-u)^nx_0^{3n}\bx^{(3n^2-n)\delta}\frac{(u^{-1}x_0;\bx^{2\delta})_n}{(ux_0;\bx^{2\delta})_n}
\end{align*}
and if $\ell(\sigma) = 2n-1$ then
\begin{align*}
1 \vert_{\omega_1} \sigma = -(-u)^nx_0^{3n-2}\bx^{(3n^2-5n+2)\delta}\frac{(u^{-1}x_0;\bx^{2\delta})_n}{(ux_0;\bx^{2\delta})_n}.
\end{align*}

Assume $\sigma\alpha_1<0$. If $\ell(\sigma) = 2n$ then
\begin{align*}
1 \vert_{\omega_1} \sigma = (-u)^nx_1^{3n}\bx^{(3n^2-2n)\delta}\frac{(u^{-1}x_1\bx^{\delta};\bx^{2\delta})_n}{(ux_1\bx^{\delta};\bx^{2\delta})_n}
\end{align*}
and if $\ell(\sigma) = 2n+1$ then
\begin{align*}
1 \vert_{\omega_1} \sigma = -(-u)^nx_1^{3n+2}\bx^{(3n^2+2n)\delta}\frac{(u^{-1}x_1\bx^{\delta};\bx^{2\delta})_n}{(ux_1\bx^{\delta};\bx^{2\delta})_n}.
\end{align*}
Changing variables $n \to -n$ in the third and fourth equalities gives
the formula stated above.
\end{proof}

\subsection{The twisted multiple Dirichlet series attached to \texorpdfstring{$\affa{1}$}{A1}}
\label{sec:a1globalcomputations}
Based on the results of the previous subsections, we compute explicitly the quadratic twisted Weyl group multiple Dirichlet
series for $\affa{1}$ over the function field $\mathbb{F}_q(t)$, with
$q\equiv 1 \pmod{4}$. We consider the twisted MDS $\Zc(\bs; \bc)$ defined in~\eqref{eq:globalwmds}, for $\bs=(s_0, s_1)$
and $\bc=(c_0,c_1)$, with monic twisting parameters $c_0, c_1 \in \mathbb{F}_q[t]$.
We assume throughout that $c_0, c_1$ are square-free and coprime. This condition is satisfied in practice, e.g. by the twisting parameters needed for the sieving procedure in the moment problem~\cite{DiaconuTwiss2023Secondary}. In the untwisted case ($c_0=c_1=1$), we also use the notation $\Zc(\bs)$.

Since the generalized Cartan matrix for $\affa{1}$ has even off-diagonal entries,
the coefficients $H$ in Section~\ref{sec:WMDS} are multiplicative:
\[
H(\bm; \bc)=\prod_{\pi^{k_0}\| m_0, \pi^{k_1}\| m_1}
H\big(\pi^{k_0}, \pi^{k_1}; \pi^{v_\pi(c_0)}, \pi^{v_\pi(c_1)} \big)
\chi_{c_0^{(\pi)}} (\pi^{k_0}) \chi_{c_1^{(\pi)}} (\pi^{k_1})
\]
where $\chi_a(b)=\symb{a}{b}$ and $a^{(b)}$ denotes the part of $a$ coprime to $b$.

Since $c_0$ and $c_1$ are coprime and square-free, the $\pi$-parts appearing in~\eqref{eq:p_parts}
are among $Z_0$, $Z_{\om_i}$, $i=0,1$. In view of the local-to-global principle later on,
we replace these $\pi$-parts by normalized versions $\wZ_\om$, given by
\begin{align*}
    \wZ_{0}(\bx;u) = \frac{Z_0(\bx;u)}{(u^2\bx^{\delta};\bx^{2\delta})_{\infty}}=\frac{1}{(ux_0;\bx^{\delta})_{\infty}(ux_1;\bx^{\delta})_\infty}
\end{align*}
and $\wZ_{\omega_i}(\bx;u)=Z_{\omega_i}(\bx;u)$ for $i=0,1$. Let $\wZc(\bs;\bc)$ denote the
MDS with $\pi$-part given by $\wZ_0(|\pi|^{-\bs}; |\pi|^{-1/2} )$, namely with $u=|\pi|^{-1/2}$ in the definition of the action, corresponding to $q=|\pi|^{-1}$ in~\eqref{eq:p_parts}. Also note that in the
$\pi$-part the $u$ parameter may differ by a sign from the choice $g_1=\gamma_1(\pi)|\pi|^{-1}$ made in~\eqref{eq:p_parts}.

Because of the multiplicativity property of the coefficients $H$, the MDS $\wZc(\bs; \bc)$ decomposes as an Euler product:
\begin{align*}
\wZc(\bs;\bc) = \prod_{\pi}\wZ_{\om_\pi(\bc)}(|\pi|^{-s_0}\chi_{c_0^{(\pi)}}(\pi), |\pi|^{-s_1}\chi_{c_1^{(\pi)}}(\pi);|\pi|^{-1/2})
\end{align*}
where the product is over monic irreducible polynomials $\pi$,
and $\om_\pi(\bc)=v_\pi(c_0)\omega_0+v_\pi(c_1)\omega_1$.

For $c\in \F_q[t]$ monic square-free, the Dirichlet series
\begin{align*}
L(s,\chi_c) =\sum_{\substack{m\in \F_q[t]\\ m \text{ monic} }} \frac{\chi_c(m)}{|m|^s} = \prod_{\pi}\frac{1}{1-\chi_c(\pi)|\pi|^{-s}},
\end{align*}
is in fact a polynomial in $q^{-s}$ of degree $\deg c-1$ when $c\ne 1$,
while for $c=1$ it equals $\zeta_{\F_q[t]}(s)=(1-q^{1-s})^{-1}$.
From the discussion above, we obtain the following result.
Recall that $\dd(\bs)=s_0+s_1$.
\begin{theorem}
\label{thm:explicitaffa1mds}
Assume $c_0, c_1\in \F_q[t]$ are coprime, monic, square-free polynomials.
We have
\begin{align*}
\wZc(\bs;c_0,c_1) &= \prod_{\substack{n\geq 0\\ n\ \mathrm{even}}}L\left(\tfrac{1}{2}+s_0+n\delta(\bs), \chi_{c_0}\right)L\left(\tfrac{1}{2}+s_1+n\delta(\bs), \chi_{c_1}\right)\\
&\cdot\prod_{\substack{n \geq 0\\ n\ \mathrm{odd}}}L\left(\tfrac{1}{2}+s_0+n\delta(\bs), \chi_{c_1}\right)L\left(\tfrac{1}{2}+s_1+n\delta(\bs), \chi_{c_0}\right)
\end{align*}
for $\Real(\delta(\bs))>0$, except for simple poles among $q^{1/2-s_i-n\dd(\bs)}=1$ when $c_0$ or $c_1$ equals 1.
\end{theorem}
\begin{proof}
The explicit formulas for $\wZ_0$, $\wZ_{\om_i}$ give the formula above. The region of convergence follows from the fact that $L(s,\chi_c)$ are polynomials in $q^{-s}$ for $c\ne 1$.
\end{proof}
\begin{remark}
The region of convergence in the theorem above is the  full interior of the complexified Tits cone $\XC^\circ$, extending to the largest possible domain the region of convergence in Theorem~\ref{thm:fullglobalconvergence}.
\end{remark}

In the untwisted case, we have the following local-to-global principle, as in the case of finite $\Phi$.
\begin{corollary}
\label{cor:untwistedlocaltoglobal}
Let $c_0=c_1=1$. Then
\begin{align*}
\wZc(\bs) = \wZ_0(q^{-s_0},q^{-s_1};q^{1/2}).
\end{align*}
\end{corollary}
\begin{proof}
This follows immediately from Theorem~\ref{thm:explicitaffa1mds} and the formula
for $\zeta_{\F_q[t]}(s)$.
\end{proof}
The choice of normalizing factor in $\wZ_0$ was made precisely so that the local-to-global
principle holds. It is interesting to note that the normalizing factor for the untwisted Chinta-Gunnells average corresponding to the root system $\wD_4$ is exactly the square of the one given here \cite[Sec. 5]{diaconu2023quadraticweylgroupmultiple}.

\subsection{Extra functional equation satisfied by twisted MDS}

We now show that the multiple Dirichlet series $\wZc(\bs;\bc)$ for $\wA_1$
from the previous subsection also satisfies
an extra functional equation of the type satisfied by its $\pi$-parts in
\S\ref{subsec:untwistedaffa1computation} and \S\ref{subsec:twistedaffA1}. We assume, as in the previous
section, that $c_0$ and $c_1$ are coprime, monic, square-free polynomials.

In the untwisted case, Corollary~\ref{cor:untwistedlocaltoglobal} already shows that $\wZc(\bs)$
is invariant under the action $|_0\tau$ defined in \S\ref{subsec:untwistedaffa1computation}.
For general $\bc$, the situation is complicated by the fact that the coefficients
of $\wZc(\bs;\bc)$ as a power series in $\bx=q^{-\bs}$ need not be polynomials in $u$.
Indeed, for $\deg c\ge 3$, the Dirichlet polynomials $L(s,\chi_c)$ in Theorem~\ref{thm:explicitaffa1mds} have coefficients which are not polynomial in $u$.
Instead, we use the fact that these coefficients are $q$-Weil numbers, as we now recall.

Let $C=C_d$ be the hyperelliptic curve $y^2=d(t)$ for  $d \in \mathbb{F}_q[t]$ monic square-free.
Its zeta function is
\begin{align*}
Z_{C}(t) = \frac{P_{C}(t)}{(1-t)(1-qt)},\qquad  (t=q^{-s}),
\end{align*}
with $P_C(t) \in \mathbb{Z}[t]$ of degree $2g$, and constant term $P_C(0) = 1$, where $g$ is the genus of $C$. Furthermore, it is a consequence of the Riemann hypothesis for curves over finite fields (proven by Weil for nonsingular projective curves  \cite{weil1948courbes}) that one may write $P_C(t) = \prod_{i=1}^{2g}(1-\gamma_i t)$ where the $\gamma_i$ are $q$-Weil numbers of weight $1$. Recall that a $q$-Weil number of weight $m\in \Z$ is an algebraic number all of whose Galois conjugates have absolute value $q^{m/2}$.
Thus writing $P_C(t) = \sum_{j=0}^{2g}a_j(q)t^j$ to emphasize that $C$ is defined over $\F_q$, we see that
\be \label{eq:coeffs}
a_j(q)=(-1)^j\sum_{|I|=j} \prod_{i\in I} \gamma_i
\ee
is a sum of $q$-Weil numbers of weight $j$, where the sum is over all subsets $I\subset \{1,\ldots, 2g\}$ with $j$ elements.

We may also view the same $C=C_d$ as being defined over an extension $\F_{q^e}$ with $e\ge 1$. From the
representation of $P_{C/\F_q}(t)$ as the characteristic polynomial of the geometric Frobenius $F$ defined over $\F_q$,
acting on \'etale cohomology, we have
\[
\begin{aligned}
 P_{C/\F_{q^e}}(t)&=\det(I-t F^{e}|H^1_{\text{\'et}}(C/\ov{\F}_q, \ov{\mathbb{Q}}_\ell)  )\\
 &=\prod_{i=1}^{2g}(1-\gamma_i^e t),
 \end{aligned}
\]
for $t=q^{-es}$ where $\ell\nmid q$ is a prime.
Therefore the coefficients $a_j(q^e)$ of  $P_{C/\F_{q^e}}(t)$ satisfy the same relation as~\eqref{eq:coeffs} with $\gamma_i$
replaced by $\gamma_i^e$.

For $d\ne 1$ as above, the $L$-function $L(s,\chi_d)$ is related to the numerator of $Z_{C_d}$ by
\be\label{eq:Lchi}
 L(s,\chi_d) = (1 -q^{-s})^{\varepsilon_{d}}P_{C_d}(q^{-s}),
\ee
where $\varepsilon_{d}$ is $0$ or $1$ according as $\deg(d)$ is odd or even \cite[\S 2]{Diaconu2018ModuliOH}.
Denote by $L(s,\chi_d;q^e)$ the Dirichlet series of $\chi_d$ with the sum over monic polynomials in $\F_{q^e}[t]$.
It is a polynomial in $q^{-es}$, and by~\eqref{eq:Lchi} its coefficients are sums of $q$-Weil
numbers satisfying the same compatibility as those of $P_{C/\F_{q^e}}(t)$. We formalize this compatibility
in the following definition, noting that a similar compatibility condition is satisfied by the coefficients of the
untwisted MDS constructed geometrically in~\cite{Diaconu2018ModuliOH,Sawin2024GeneralMD}.

\begin{definition}
The power series $f(\bx;q^e)=\sum_{\lam\in Q^+} a_\lam(q^e) \bx^{e \lam}$ defined for $e\ge 1$
form a compatible system if their coefficients decompose as follows:
\be\label{eq:compatible}
a_\lam(q^e)=\sum_{j\in I_\lam} c_{\lam,j} \gamma_{\lam,j}^e,
\ee
with finite index sets $I_\lam$, integers $c_{\lam, j}$ and $q$-Weil numbers  $\gamma_{\lam,j}$ of weight $m_{\lam, j}$
such that $\gamma_{\lam,j}\ne\gamma_{\lam,k} $ if $j\ne k$.

For such a compatible system and $k\ge 1$, we define
$f^{\tau^k}(\bx;q)$ to be the formal series obtained by replacing each coefficient
$a_\lam(q)$ in the expansion of $f(\bx;q)$ by $\sum_{j\in I_\lam} c_{\lam,j} \gamma_{\lam,j} \bx^{k m_{\lam,j} \dd}$.
\end{definition}

For fixed $\bc=(c_0,c_1)\in \F_q[t]^2$ with $c_0, c_1$ monic square-free,
let $\wZc(\bs;\bc;q^e)$ be the  multiple Dirichlet series defined like  $\wZc(\bs;\bc;q):=\wZc(\bs;\bc)$, but over $\F_{q^e}$.
By Theorem~\ref{thm:explicitaffa1mds} and the preceding discussion, these series form
a compatible system, viewed as power series in  $\bx=q^{-\bs}$.

Although we will need only the definition of $f^{\tau}(\bx;q)$, the existence of a compatible system
as above guarantees the uniqueness of the decomposition~\eqref{eq:compatible} for $a_\lam(q)$.
Note that the resulting series $f^\tau$ may not be well-defined,
e.g. for $f(\bx;q)=\sum_{n \ge 0} q^{-n/2} \bx^{n\dd}$.
However we will see in the proof of Theorem~\ref{thm:tauactiononmds}
that it is well-defined for the compatible system formed by $\wZc(\bs;\bc;q^e)|_{q^{-\bs}\mapsto \bx}$.

We extend the actions defined in \S\ref{subsec:untwistedaffa1computation} and \S\ref{subsec:twistedaffA1}
to a power series $f(\bx)=f(\bx;q)$ as in the previous definition:
\[
f\vert_0\tau(\bx)=\frac{f^{\tau}(\bx)}{(1-ux_0)  (1-ux_1)}, \quad
f\vert_{\om_1}\tau^2(\bx)=\frac{f^{\tau^2}(\bx)}{(1-ux_0)  (1-ux_1\bx^\dd)},
\]
where $u=\sqrt q$, with $f\vert_{\om_0}\tau^2$ defined by symmetry.
Note that $\wZ_{0}$ and $Z_{\om_i}$ are invariant under the corresponding actions
by the results in \S\ref{subsec:untwistedaffa1computation} and \S\ref{subsec:twistedaffA1}.
We also define
$$f \vert_{\omega_0+\omega_1} \tau = f^{\tau},$$
so that the function $Z_{\om_0+\om_1}$, which equals 1 by Macdonald's formula,
is also invariant under this action.

For a twisting parameter $\om=l_0\om_0+l_1\om_1\in P^+$, we denote by
$\ov{\om}\in\{0, \om_0, \om_1, \om_0+\om_1\}$ its reduction modulo 2.
Since the Chinta-Gunnells actions with twists $\om$ and $\ov{\om}$ differ only by monomials in $x_i$,
it follows that  the Weyl group action $|_\om$ commutes with the actions $|_{\ov{\om}} \tau$
or $|_{\ov{\om}}\tau^2$ defined above, by the results of \S\ref{subsec:untwistedaffa1computation}, \S\ref{subsec:twistedaffA1} (the case $\ov{\om}=\om_0+\om_1$ is trivial since in this case the Chinta-Gunnells action does not depend on $u$).

By Theorem~\ref{thm:fullglobalconvergence} (or by the explicit formula in Theorem~\ref{thm:explicitaffa1mds} and the functional equation
of $L(s,\chi_{c_i})$), the multiple Dirichlet series $\wZc(\bs;\bc)$ is invariant
under the Chinta-Gunnells action with twisting parameter
$\om_\bc:=\deg c_0\cdot \om_0+\deg c_1\cdot \om_1$,
after replacing $|m_i|^{-s_i}$ by $x_i^{\deg m_i}$ in its definition~\eqref{eq:globalwmds}
and taking $u=\sqrt q$ in the definition of the action. The main result of this section
shows that $\wZc(\bs;\bc)$ also satisfies an extra functional equation under the action $|_{\ov{\om}_\bc} \tau$, or  $|_{\ov{\om}_\bc} \tau^2$,
just like its $\pi$-parts.

\begin{theorem}
\label{thm:tauactiononmds}
Let $c_0, c_1 \in \mathbb{F}_q[t]$ be square-free and coprime and set $n_i=\deg c_i$ and
$\om_{\bc}=n_0 \om_0+n_1\om_1$.  Let $\Zm(\bx; \bc)$ denote
the MDS $\wZc(\bs;c_0,c_1)$ after changing variables $q^{-s_i}\mapsto x_i$.

We then have the functional equations:
\[
\Zm(\bx;\bc)=\Zm(\bx;\bc)\vert_{\ov{\om}_\bc} \tau\qquad
\text{if  $n_0=n_1=0$ or if $n_0$, $n_1$ odd};
\]
\[
\Zm(\bx;\bc)=\Zm(\bx;\bc)\vert_{\ov{\om}_\bc} \tau^2 \cdot
\begin{cases}
 1 & \text{ if $n_0=0$ and $n_1$ odd}\\
 q\bx^{2\dd} & \text{ if $n_0\ne 0$ even and $n_1=0$ or $n_1$ odd }  \\
 q^2\bx^{4\dd} & \text{ if $n_0$, $n_1$ nonzero even.  }
\end{cases}
\]
\end{theorem}
By symmetry, the theorem covers all possible cases.
\begin{proof}
Using the relation~\eqref{eq:Lchi}, we express each $L$-function $L(\tfrac12+s,\chi_d)$ in the product in  Theorem~\ref{thm:explicitaffa1mds}
in terms of $P_{C_d}(q^{-\tfrac12-s})$, and denote the product involving the polynomials $P_{C_d}$ by  $\Zm_\geom(\bx;\bc)$, after the change of
variables $x_i=q^{-s_i}$ (for $d=1$ we set $P_{C_d}=1$).  We obtain
\[
\Zm(\bx;\bc)=\Zm_\geom(\bx;\bc)\Delta_{n_0}(x_0,x_1;u) \Delta_{n_1}(x_1,x_0;u),
\]
where $u=\sqrt q $ and for $n\ge 0$ we define:
\[\Delta_n(x_0,x_1;u)=\begin{cases} 1 & \text{if $n$ odd}  \\
            \qpoc{u^{-1}x_0}{\bx^{2\dd}}\cdot \qpoc{u^{-1}x_1\bx^\dd}{\bx^{2\dd}} & \text{if $n\ne 0$ even}  \\
             \qpoc{ux_0}{\bx^{2\dd}}^{-1}\cdot \qpoc{ux_1\bx^\dd}{\bx^{2\dd}}^{-1}   & \text{if $n=0$}.
                    \end{cases}
\]
Note that $\Zm_\geom^\tau=\Zm_\geom$, because each Dirichlet polynomial $P_{C_d}(q^{-\tfrac12-s})$ has coefficients that are
$q$-Weil numbers of weight 0; in particular, this proves that the series $\Zm^{\tau^k}(\bx;\bc)$ are well-defined for $k=1,2$.

Therefore the ratios $\Zm/\Zm^\tau$ and $\Zm/\Zm^{\tau^2}$ can be easily computed from the transformation
formula for the Pochhammer symbols above under $u\mapsto u\bx^\dd$.
Comparing with the definition of the corresponding actions $|_{\ov{\om}_\bc}\tau$ or $|_{\ov{\om}_\bc}\tau^2$
(which depend only on the parity of $n_0$ and $n_1$)   finishes the proof.
\end{proof}

\bibliographystyle{amsalpha}
\bibliography{bibliography1.bib}

\Addresses

\end{document}